\documentclass[11pt, a4paper]{article}

\usepackage{geometry}
\usepackage{amsmath}
\usepackage{amsthm}
\usepackage{amssymb}
\usepackage{bbm} 
\usepackage{color}
\usepackage{xcolor}
\usepackage{graphicx}
\usepackage[toc,page]{appendix}
\usepackage{authblk}
\usepackage[unicode,colorlinks=true,linkcolor=black,citecolor=black,urlcolor=black,pagebackref=false]{hyperref} 

\newcommand{\R}{\mathbb{R}}
\newcommand{\1}{\mathbbm{1}}
\renewcommand{\P}{\mathbb{P}}

\newcommand{\E}{\mathbb{E}}

\DeclareMathOperator\supp{supp}
\DeclareMathOperator\id{id}

\DeclareMathOperator*{\argmin}{arg\,min}

\newtheorem{theorem}{Theorem}
\newtheorem{proposition}[theorem]{Proposition}
\newtheorem{lemma}[theorem]{Lemma}
\newtheorem{corollary}[theorem]{Corollary}

\newtheorem{assumption}[theorem]{Assumption}
\theoremstyle{definition}
\newtheorem{example}[theorem]{Example}
\newtheorem{remark}[theorem]{Remark}

\title{Small-time central limit theorems for stochastic Volterra integral equations and their Markovian lifts}

\author[1]{Martin Friesen \thanks{Email: martin.friesen@dcu.ie}}
\author[2]{Stefan Gerhold \thanks{Email: sgerhold@fam.tuwien.ac.at}}
\author[2]{Kristof Wiedermann \thanks{Email: kristof.wiedermann@tuwien.ac.at}}
\affil[1]{\small School of Mathematical Sciences, Dublin City University}
\affil[2]{\small Institute of Statistics and Mathematical Methods in Economics, TU Wien}

\date{\today}
\numberwithin{equation}{section}
\numberwithin{theorem}{section}

\begin{document}

\maketitle

\begin{abstract}
\noindent We study small-time central limit theorems for stochastic Volterra integral equations with H\"older continuous coefficients and general locally square integrable Volterra kernels. We prove the convergence of the finite-dimensional distributions, a functional CLT, and limit theorems for smooth transformations of the process, which covers a large class of Volterra kernels that includes rough models based on Riemann-Liouville kernels with short- and long-range dependencies. To illustrate our results, we derive asymptotic pricing formulae for digital calls on the realized variance in three different regimes. The latter provides a robust and model-independent pricing method for small maturities in rough volatility models. Finally, for the case of completely monotone kernels, we introduce a flexible framework of Hilbert space-valued Markovian lifts and derive analogous limit theorems for such lifts.

\end{abstract}
\vspace{0.2cm}
{\small \textbf{Keywords:} stochastic Volterra integral equation; central limit theorem; completely monotone kernel; Markovian lift; volatility derivatives.\vspace{0.2cm}\newline
}

\section{Introduction}

In recent years, stochastic Volterra integral equations (SVIEs) have been a very active research area, often motivated by the fast-growing field of rough volatility models (see \cite{rvol, pricingroughvol, roughhestcharfct, roughhestonhedging, volisrough}). The rough Heston model $(S,v)$ provides one of the most prominent examples of such. There $S$ denotes the asset price process given by
\[
    \mathrm{d}S_t = \mu\hspace{0.03cm} S_t\, \mathrm{d}t + \sigma \sqrt{v_t}\hspace{0.03cm}S_t \,\mathrm{d}W_t, \quad t \in \R_{+},
\]
for some Brownian motion $W$, and $v \geq 0$ models the instantaneous variance process as a rough CIR process, i.e.\ it is the unique nonnegative weak solution of
\begin{equation}\label{eq:roughHeston}
    v_{t}=v_{0}+\int_{0}^{t}\frac{(t-s)^{H-1/2}}{\Gamma(H+\frac{1}{2})}\hspace{0.03cm}\kappa\hspace{0.03cm}(\theta-v_{s})\,\mathrm{d}s + \int_{0}^{t}\frac{(t-s)^{H-1/2}}{\Gamma(H+\frac{1}{2})}\hspace{0.03cm}\xi\hspace{0.03cm}\sqrt{v_{s}}\,\mathrm{d}B_{s}, \quad t \in \R_{+},
\end{equation}
where $B$ is another Brownian motion correlated with $W$ (see \cite[Theorem 6.1]{AbiJaLaPu19}). Based on an approximation procedure by nearly unstable Hawkes processes, pricing and hedging of derivatives were studied in \cite{roughhestcharfct, roughhestonhedging}. More generally, various extensions to (multi-dimensional) affine and non-affine Volterra processes have been studied in \cite{AbiJaLaPu19, rvol}. Thus, from a general perspective, in this work, we consider convolution-type stochastic Volterra integral equations (SVIEs) of the form
\begin{equation}\label{eq:generalSVIE}
  X_t = x_{0}+\int_{0}^{t}K(t-s) \hspace{0.03cm}b(X_{s})\, \mathrm{d}s + \int_0^t K(t-s)\hspace{0.03cm}\sigma(X_s)\, \mathrm{d}B_s,\quad t \in \R_{+},
\end{equation} 
where $B$ denotes a Brownian motion on some filtered probability space $(\Omega,\mathcal{F},\mathbb{F},\mathbb{P})$, $x_{0}\in~\R$ the initial condition, $b,\sigma:\R\rightarrow\R$ the drift and diffusion coefficients, and $K$ the Volterra kernel. Let us stress that the coefficients in \eqref{eq:roughHeston} are merely $1/2$-H\"older continuous. To cover such cases, here and below we shall work under the following set of assumptions.
\begin{assumption}\label{assumption:coefficientassumptions}
    There exist a constant $C > 0$ and Hölder exponents $\chi_{b}$, $\chi_{\sigma}\in (0,1]$ such that for all $x,y \in \R$:
\begin{equation}\label{eq:lingrowthandHölder}
\begin{aligned}
    |b(y)-b(x)|\le C\,|y-x|^{\chi_{b}} \hspace{0.3cm}\mbox{and}\hspace{0.3cm}|\sigma(y)-\sigma(x)|\le C\,|y-x|^{\chi_{\sigma}}.
\end{aligned}
\end{equation}
 Moreover, for the Volterra kernel $K \in L_{\mathrm{loc}}^2(\R_+)$ there exist constants $\gamma,\gamma_{*} > 0$ and $C, C^{*}~>~0$ with
\begin{equation}\label{eq:kernelL2order}
    C^{*} t^{2\gamma_*}\leq\int_0^t K(s)^2\, \mathrm{d}s \leq C t^{2\gamma}, \qquad t \in (0,1].
\end{equation}
\end{assumption}

Note that the restriction onto $t \in (0,1]$ in \eqref{eq:kernelL2order} is only for convenience since our primary interest lies in the study of small-time central limit theorems. In any case, all results remain valid if we replace the condition with $t \in (0,T]$ and some fixed time horizon $T > 0$. Note that $b,\sigma$ have linear growth by \eqref{eq:lingrowthandHölder}, $\gamma_*\ge\gamma$ and $K \not\equiv 0$ due to \eqref{eq:kernelL2order}. Moreover, in case of singular kernels, where $K$ may only be defined pointwise on $\R_+^* :=\R_+\setminus \{0\}$, $K$ is still well-defined on $\R_+$ in a $L_{\mathrm{loc}}^2$-sense. 

In general, Assumption \ref{assumption:coefficientassumptions} does not guarantee the existence of a solution to \eqref{eq:generalSVIE}. For Lipschitz continuous coefficients $b, \sigma$, the existence of a unique strong solution was shown in \cite[Theorem 3.3]{AbiJaLaPu19}, see also \cite{berger_mizel, protter_volterra, zhangSVE}. For H\"older continuous coefficients and under slight additional assumptions, the existence of weak solutions to \eqref{eq:generalSVIE} was established in \cite{AbiJaLaPu19} for Volterra kernels that admit a resolvent of the first kind\footnote{The resolvent of the first kind is a measure $R\in \mathcal{M}_{\mathrm{loc}}(\R_{+};\R)$ of locally bounded total variation with $(R \ast K)_{t} = (K\ast R)_{t} =1$ for every $t\in\R_{+}$, see \cite[Definition 5.5.1]{grippenberg}.}, satisfy Assumption \ref{assumption:coefficientassumptions}, and the H\"older increment condition
\begin{equation}\label{eq:kernelincrementL2}
    \int_0^T |K(s+h) - K(s)|^2\, \mathrm{d}s \leq \overline{C}_Th^{2\overline{\gamma}}, \qquad h \in (0,1]
\end{equation}
holds for some $\overline{\gamma}>0$ and a suitable constant $\overline{C}_T > 0$ depending on $T>0$. The latter is satisfied by most of the kernels considered in applications (see \cite[Example 2.3]{AbiJaLaPu19}). Weak existence results that are not based on the resolvent of the first kind were obtained in~\cite{proemel_scheffels_weak} for kernels not necessarily of convolution type, while an extension towards equations with jumps has been considered in \cite{weak_solution}. Finally, for non-Lipschitz coefficients with possibly singular kernels, the pathwise uniqueness of solutions is more subtle and yet not fully understood, see \cite{hamaguchi2023weak, mytnik, proemel_dec22} for some recent results in this direction. 

Consequently, it follows from \cite{weak_solution, AbiJaLaPu19, proemel_scheffels_weak} that weak existence to \eqref{eq:generalSVIE} can be established for H\"older continuous $b,\sigma$ and a large class of Volterra kernels that covers many interesting examples including the Riemann-Liouville kernel $K_{H}(t)=c(H)\,t^{H-1/2}$ with parameter $H>0$, the gamma kernel $K_{\beta, H}(t)=c(H)\,t^{H-1/2}\,e^{-\beta t}$ with parameters $\beta>0$ and $H>0$, where in both cases $c(H)>0$ denotes a constant depending on $H$, the mixed exponential kernel $K_{c,\lambda}(t)=\sum_{i=1}^{n}c_{i}e^{-\lambda_{i} t}$ with $c_{i}>0$ and $\lambda_{i}\ge 0$ for every $i\in\{1,\dots,n\}$, and the Riemann-Liouville kernel modulated by the logarithm given by $K(t) = t^{H-1/2}\log(1+t^{-\alpha})$ with $\alpha \in (0,1]$ and $H \in (0,1/2]$.

Large deviations for stochastic Volterra integral equations (SVIEs) have received a lot
of attention, largely because of their applications to rough volatility models, see
\cite{MR4657189, fordegerholdsmith, MR4401811, MR4543015, MR4698114, MR4410038}. While our motivation stems only partially from mathematical finance, it seems natural to complement our knowledge of large deviations with central limit theorems. For classical stochastic differential equations, small-time central limit theorems with H\"older continuous coefficients have been studied in \cite[Corollary 4.1]{GaWa16}, while a functional CLT (fCLT) was obtained in \cite{GeKlPoSh15}. For Volterra processes, the situation currently is much less developed.
In~\cite{Du19}, a CLT for SVIEs driven by fractional Brownian motion is proven; the Hurst
parameter is in $(1/2,1)$, so that integration can be defined pathwise.
A general account on small-time CLTs for~\eqref{eq:generalSVIE} seems to be absent from the literature. Finally, a small-noise CLT for SVIEs is presented in~\cite{Qi23}, while~\cite{IzKh24} contains some related results on the law of the iterated logarithm. Thus, in this work, we complement the known state-of-the-art and establish in Theorem \ref{theorem:varyinginitialCLT} a general small-time central limit theorem under Assumption~\ref{assumption:coefficientassumptions}, study the fCLT in Corollary~\ref{corollary:functionalCLT}, and finally consider in Corollary~\ref{corollary:c2trafosfinitedimCLT} a CLT for the process $f(X)$ where $f\in C^1(O)$ for some open set $O\subseteq \R$ containing $x_0$. 

Let us remark that an alternative method to prove an fCLT for $X$ could be based on the following observation suggested to us by Masaaki Fukasawa. Namely, let $X$ be given as in \eqref{eq:generalSVIE}, and define
the semimartingale
\[
  \mathrm{d}M = b(X)\,\mathrm{d}t + \sigma(X)\,\mathrm{d}W.
\]
Then~$X$ has a representation as a fractional integral of the form $X = K \ast \mathrm{d}M$. Since fractional integration is continuous from H\"older space to H\"older space (see Section~3 in~\cite{SaKiMa93} and Appendix~A in~\cite{FuUg23}), a functional CLT for~$M$ should yield an fCLT for~$X$. However, since the coefficients of $M$ depend on $X$, the latter would require a Hölder fCLT for general semimartingales, which is not yet available in the literature (compare with \cite{GaWa16, GeKlPoSh15} for related results in this direction).  From this perspective, our approach is based on a direct study of the corresponding SVIEs.

One typical application of small-time limit theorems in mathematical finance is studying the price asymptotics of derivatives which are close to at-the-money (ATM) (cf.~\cite{FrGePi18} and \cite[Section~4]{GeKlPoSh15}). As solutions to \eqref{eq:generalSVIE} are in general not semimartingales, SVIEs are typically incorporated in asset price modelling via the non-tradeable variance process~$v$ (cf.~\eqref{eq:roughHeston}). In this framework, our functional CLT obtained in Corollary \ref{corollary:functionalCLT} gives insights into the small-time behaviour of options on the realized variance~$v$ (see~\cite{LaMuSt21}). In particular, we can investigate prices of digital calls on the average realized variance on $[0,T]$, i.e.\ $\mathcal{V}_T:=T^{-1}\int_{0}^{T}v_t\,\mathrm{d}t$, in the ATM case, a regime called ``almost  ATM'' (AATM) as discussed in~\cite{FrGePi18}, and its boundary case. In the latter regime, Proposition \ref{proposition:volatilityderivatives} yields for the rough Heston model (see~\eqref{eq:roughHeston}) for every $a\in\R$ the asymptotic price formula
 \begin{equation}\label{eq:volatilityyderivativesRHeston} \lim_{n\rightarrow\infty}\mathbb{E}\big[\mathbbm{1}_{\{\mathcal{V}_{1/n}\ge v_{0} + n^{-H}a\}}\big]=1-\Phi\left(\sqrt{2H+2}\left(H+\tfrac{1}{2}\right)\Gamma\left(H+\tfrac{1}{2}\right)\frac{a}{\xi\sqrt{v_0}}\right),
 \end{equation}
 showcasing in particular high robustness in the model parameters. Finally, note that we do not consider higher order terms beyond the limit price. The former are relevant when studying the implied volatility skew (see~\cite{GeGuPi16,GeKlPoSh15,JaTo20}).

As another application for the CLTs developed in this work, in a companion paper we provide a general proof that solutions of SVIEs with non-degenerate kernels do not possess the Markov property. The latter became a folklore fact stated in many works on SVIEs, albeit, to the best of our knowledge, a rigorous proof is still open.

Since stochastic Volterra processes \eqref{eq:generalSVIE} are, in general, neither semimartingales nor satisfy the Markov property, as mentioned above, a common approach to overcome these obstacles is based on an augmentation of the state space that allows to recover the Markov property in an infinite-dimensional framework, see \cite{markovian_structure, MR4181950, MR4503737, FH24}. One commonly used approach is based on the study of such processes in terms of Laplace transforms for completely monotone kernels, see \cite{markovian_structure, hamaguchi2023weak, H23, FH24}. Inspired by this idea, we follow the exposition \cite{H23}, which applies to completely monotone Volterra kernels that satisfy a minor small-time integrability condition, to study Markovian lifts of the equation
\begin{align}\label{eq: 6intro}
    X_t = g(t) + \int_0^t K(t-s)\hspace{0.02cm}b(X_s)\, \mathrm{d}s + \int_0^t K(t-s)\hspace{0.02cm}\sigma(X_s)\, \mathrm{d}B_s,\quad t\in\R_+.
\end{align}
In contrast to \cite{H23}, we modify the state-space in such a way that it also covers constant initial curves $g \equiv x_0$, i.e.~\eqref{eq:generalSVIE}, independently of the support of the Bernstein measure of~$K$. Complementing the construction given in \cite{H23}, we prove in Theorem \ref{thm: Markovian lift Markov property} the existence of a continuous weak solution for the stochastic evolution equation 
\begin{align*}
 \mathcal{X}_t = S(t)\xi_g + \int_0^t S(t-s)\hspace{0.02cm}\xi_K \hspace{0.02cm}b(\Xi\mathcal{X}_s)\, \mathrm{d}s + \int_0^t S(t-s)\hspace{0.02cm}\xi_K\hspace{0.02cm} \sigma(\Xi\mathcal{X}_s)\, \mathrm{d}B_s, \quad t\in\R_+,
\end{align*}
i.e.\ the corresponding Markovian lift $\mathcal{X}$, such that the associated projection operator $\Xi$ is applicable. The latter relates the SVIE to its lift via the identity $\Xi\mathcal{X}=X$. The central argument therein relies on the proof that each continuous weak solution of~\eqref{eq: 6intro} may be used for defining a continuous Markovian lift in a weighted Hilbert space of sufficient regularity. Let us remark that the class of kernels we consider here satisfies both Assumption~\ref{assumption:coefficientassumptions} and~\eqref{eq:kernelincrementL2}, as shown in Lemma \ref{lemma:liftkernelL2estimates}.

We conclude this paper by deriving small-time central limit theorems also for transformations of the Markovian lift $\mathcal{X}$ under continuous linear functionals. More precisely, in Theorem \ref{theorem:LiftClt} we first prove a version for the finite-dimensional distributions and, under some minor technical assumptions, we then proceed to show a functional CLT in Theorem \ref{theorem:LiftfCLT}. Finally, we consider several examples which illustrate that, on the one hand, for the continuous linear functional $\Xi$ these results coincide with those obtained in Section~\ref{section:smalltimeCLTSVIE}, albeit restricted to completely monotone kernels. On the other hand, we show that our central limit theorems for Markovian lifts go beyond the results obtained in Section~\ref{section:smalltimeCLTSVIE}. For instance, they provide limit theorems for the $n$-dimensional Volterra Ito-process $(X^1,\dots, X^n)^{\intercal}$ given by 
\[
    X^j_t = g_j(t) + \int_0^t K_j(t-s)\hspace{0.02cm}b(X_s)\, \mathrm{d}s + \int_0^t K_j(t-s)\hspace{0.02cm}\sigma(X_s)\, \mathrm{d}B_s, \qquad t\in\R_+,
\]
where $X$ is a solution to \eqref{eq: 6intro}, and $K_1,\dots, K_n$ are completely monotone Volterra kernels obtained by an absolutely continuous change of the Bernstein measure with regards to the original completely monotone Volterra kernel $K$.

\subsubsection*{Notation.} 
Here and below, for arguments requiring estimates merely modulo a multiplicative constant, we denote by~$\lesssim$ an inequality up to a constant factor that is not further specified. The precise quantities on which the constant is allowed to depend are in any case clear from the context. 

\subsubsection*{Structure of the work}

In Section \ref{section:smalltimeCLTSVIE} we study the small-time CLT regime for solutions of \eqref{eq:generalSVIE} under Assumption \ref{assumption:coefficientassumptions} first for the finite-dimensional distributions, then derive an fCLT via tightness arguments, and finally consider CLTs for the transformed process $f(X)$, where $f$ is a sufficiently smooth function. We close this section with a few illustrative examples and an application to asymptotic option pricing formulae for digital options on the realized volatility in different regimes. In Section \ref{section:Liftpreliminaries} we briefly introduce a framework for Hilbert space valued Markovian lifts and discuss some auxiliary results, while the CLTs for the Markovian lift and examples thereof are studied in Section \ref{section:smalltimeCLTLift}. Finally, a few auxiliary results are collected in the appendix of this work, in particular the existence of a Markovian lift of $X$ with continuous sample paths, for which the projection operator is applicable.

\section{Small-time central limit theorems for SVIEs}\label{section:smalltimeCLTSVIE}

\subsection{CLT and functional CLT for solutions to SVIEs}\label{subsection:SVIECLT}

As a first step towards a small-time CLT for solutions to \eqref{eq:generalSVIE} under Assumption \ref{assumption:coefficientassumptions}, let us derive a small-time moment estimate on the process 
\begin{displaymath}
    Z_{t}:=\int_0^t K(t-s)\hspace{0.03cm}(b(X_s)-b(x_0))\,\mathrm{d}s+\int_0^t K(t-s)\hspace{0.03cm}(\sigma(X_s)-\sigma(x_0))\,\mathrm{d}B_s, \quad t\in\R_+,
\end{displaymath}
as given in the next proposition.
\begin{proposition}\label{corollary:zasymptotics}
    Let $X$ be a continuous solution to the SVIE \eqref{eq:generalSVIE} and suppose that Assumption~\ref{assumption:coefficientassumptions} is satisfied. Then for every $p\ge\max\{2/\chi_{b},2/\chi_{\sigma}\}$ we obtain for some $\overline{C}_p > 0$:
    \begin{equation}\label{eq:momentestimate2}
        \mathbb{E}[|Z_t|^p] \leq \overline{C}_p\, (1+|x_{0}|^{p})\left(t^{\frac{p}{2} + p\gamma (1 + \chi_b)} + t^{p\gamma (1 + \chi_{\sigma})}\right),\quad t\in [0,1].
    \end{equation}
    Moreover, for any \begin{equation}\label{eq:zexponentcorridor}
      q\in \left(\gamma_{*},\min\left\{\tfrac{1}{2}+\gamma\hspace{0.02cm}(1+\chi_{b}),\gamma\hspace{0.02cm}(1+\chi_{\sigma})\right\}\right),
  \end{equation}
  provided that the interval is non-degenerate, $z(t)=t^{q}$ satisfies
    \begin{equation}\label{eq:zasymptotics}
        \lim_{t\rightarrow 0}\frac{z(t)}{\sqrt{\mathrm{Var}\big[(K\ast \mathrm{d}B)_{t}\big]}}=0\hspace{0.3cm}\mbox{and}\hspace{0.3cm}\lim_{t\rightarrow 0}\frac{\mathbb{E}[|Z_{t}|^p]}{z(t)^p}=0.
    \end{equation}
\end{proposition} 
\begin{proof}
    Let us first note that by Lemma \ref{lemma:momentestimatewithg} applied for $g\equiv x_0$, we obtain for each $p \geq 2$:
    \begin{equation}\label{eq:momentestimateSVIE}
        \mathbb{E}[|X_t - x_0|^p] \leq C_p\,(1+|x_0|^p)\,t^{p\gamma},\quad t\in [0,1],
    \end{equation}
    where $C_p > 0$ is some constant. Hence, an application of the Jensen inequality combined with \eqref{eq:momentestimateSVIE} yields for $p\ge\max\{2/\chi_{b},2/\chi_{\sigma}\}$: 
    \begin{align*}
      \mathbb{E}[|Z_t|^p] &\lesssim \left(\int_0^t |K(s)|\, \mathrm{d}s \right)^{p-1}\int_0^t |K(t-s)|\,\E\left[|X_s-x_0|^{\chi_b p} \right]\, \mathrm{d}s
      \\ &\qquad + \left( \int_0^t |K(s)|^2\, \mathrm{d}s\right)^{\frac{p}{2} - 1}\int_0^t |K(t-s)|^2 \,\E\left[ |X_s - x_0|^{\chi_{\sigma}p} \right]\, \mathrm{d}s
      \\ &\lesssim t^{(p-1)(\gamma + \frac{1}{2})}\int_0^t |K(t-s)|\, s^{p\chi_b \gamma}\left(1 + |x_0|^p\right)\, \mathrm{d}s
      \\ &\qquad + t^{(p-2)\gamma}\int_0^t |K(t-s)|^2\, s^{p\chi_{\sigma}\gamma}\left(1 + |x_0|^p\right)\, \mathrm{d}s
      \\ &\lesssim \left(t^{\frac{p}{2} + p\gamma (1 + \chi_b)} + t^{p\gamma (1 + \chi_{\sigma})}\right)\left(1 + |x_0|^p\right). 
    \end{align*}
    For the second assertion, note that by \eqref{eq:momentestimate2}, $\mathbb{E}[|Z_{t}|^p]$ is bounded above by a function which is $\mathcal{O}\big(t^{\min\{\frac{p}{2}+p\gamma\hspace{0.02cm}(1+\chi_{b}),\,p\gamma \hspace{0.02cm}(1+\chi_{\sigma})\}}\big)$ as $t\downarrow 0$. Thus by the particular choice of $q$ given as in~\eqref{eq:zexponentcorridor}, we obtain $\lim_{t\rightarrow 0}\frac{\mathbb{E}[|Z_{t}|^p]}{z(t)^p}=0$. The other assertion follows from an application of~\eqref{eq:kernelL2order}:
    \begin{displaymath}
        \big(\mathrm{Var}\big[(K\ast \mathrm{d}B)_{t}\big]\big)^{-1/2}=\bigg(\int_{0}^{t}K(t-s)^2\,\mathrm{d}s\bigg)^{-1/2}\le \overline{C}t^{-\gamma_{*}},
    \end{displaymath}
   when taking $q>\gamma_{*}$ into account. 
\end{proof}

Note that, if $\gamma = \gamma_* > 0$, then the interval \eqref{eq:zexponentcorridor} is nondegenerate whenever $\chi_{b}, \chi_{\sigma}\in (0,1]$. Next, let us state and prove our main result on the small-time CLT and its extension to a functional CLT for the sequence of normalized processes
\begin{equation}\label{eq:varyinginitialCLTprocesssequence}
    \sqrt{\lambda(n)}\big(X_{\cdot/n}^{x_{n}}-x_{n}\big), \quad n\in\mathbb{N},
\end{equation}
where $X^{x_n}$ satisfies the SVIE \eqref{eq:generalSVIE} with initial condition~$x_n$, and $\lambda(n)$ denotes the normalization factor defined for $n\in\mathbb{N}$ by
\begin{align}\label{eq: lambda definition}
    \lambda(n) = \left(\int_0^{1/n}K(r)^2\, \mathrm{d}r\right)^{-1}.
\end{align}

Note that here we also allow the initial condition to vary. The following is our first main result on the CLT for \eqref{eq:generalSVIE}. It provides a small-time central limit theorem for the finite-dimensional distributions of \eqref{eq:generalSVIE}. Examples of kernels~$K$ and the associated limiting kernels~$\overline{K}$ in condition~(ii) of the below theorem are given in Subsection~\ref{subsection:CLTKernelExamples}. 

\begin{theorem}\label{theorem:varyinginitialCLT} 
   Suppose that Assumption \ref{assumption:coefficientassumptions} is satisfied. For each $n \geq 1$, let $X^{x_n}$ be a continuous weak solution to the associated SVIE \eqref{eq:generalSVIE} with initial condition~$x_{n}$, defined on some filtered probability space. Suppose that $x_n \longrightarrow \overline{x}$ and $\sigma(\overline{x}) \neq 0$. Fix $N\in\mathbb{N}$ and consider a family of time points $(t_{i})_{i\in\{1,\dots,N\}}$ with $0<t_{1}<\cdots<t_{N}$ such that the following conditions are satisfied:
   \begin{enumerate}
        \item[(i)] The constants $\gamma_* \geq \gamma$ from \eqref{eq:kernelL2order} satisfy
        \begin{equation}\label{eq:varyinginitialCLTexponentcorridor}
            \gamma_* < \min\big\{\gamma + \tfrac{1}{2}, \gamma\hspace{0.05cm}(1+\chi_{\sigma})\big\},
        \end{equation}
        
        \item[(ii)] There exists $\overline{K} \in L_{\mathrm{loc}}^2(\R_+)$ such that for all $0 \leq s < t \leq 1$:
        \[
        \lim_{n \to \infty} \lambda(n)\int_0^{s/n} K( (t-s)/n + r) K(r)\, \mathrm{d}r
        = \int_0^s \overline{K}(t-s + r)\overline{K}(r)\,\mathrm{d}r.
        \]
    \end{enumerate}
   Then, as $n \to \infty$, we obtain the following small-time CLT
   \begin{equation}\label{eq:varyinginitialCLT}
    \left( \sqrt{\lambda(n)}\big(X_{t_{j}/n}^{x_{n}}-x_{n}\big) \right)_{j=1,\dots,N} \stackrel{d}{\longrightarrow} \left(Y_{t_{j}}^{\sigma,\infty}\right)_{j=1,\dots, N},
   \end{equation}
   where the Gaussian limit process is given by $Y^{\sigma,\infty} = \sigma(\overline{x})\int_0^{\cdot} \overline{K}(\cdot-s)\, \mathrm{d}B_s$.
\end{theorem}
\begin{proof}
    By the multi-dimensional version of the Helly-Bray theorem, it sufficies to prove pointwise convergence of the associated distribution functions for all continuity points. For notational convenience and w.l.o.g., we write $\mathbb{P}$ instead of $\mathbb{P}^{n}$ as we study merely the distributions of the involved processes and all manipulations enabling us to reduce the problem to the Gaussian case are performed for fixed $n\in\mathbb{N}$. First, we separate the Gaussian part of each $X_{t}^{x_n}$ from its (usually) non-Gaussian remainder:
    \begin{align}
      X_{t}^{x_n} &= x_{n}+b(x_n)\int_0^t K(t-s)\hspace{0.05cm}\mathrm{d}s + \int_0^t K(t-s)\hspace{0.05cm}(b(X_{s}^{x_n})-b(x_n))\hspace{0.05cm}\mathrm{d}s \notag \\
      &\hspace{0.5cm}+ \sigma(x_n)\int_0^t K(t-s)\hspace{0.05cm}\mathrm{d}B_s + \int_0^t K(t-s)\hspace{0.05cm}(\sigma(X_{s}^{x_n})-\sigma(x_n))\hspace{0.05cm}\mathrm{d}B_s \notag\\
      &=: x_{n} + Y_t^{b,n}+Z_t^{b,n}+Y_t^{\sigma,n}+Z_t^{\sigma,n}. \label{eq:generalSVIEdividedinYZinitvarying}
    \end{align}
    Moreover, define $Y_{t}^{n}:=Y_t^{b,n}+Y_t^{\sigma,n}$ and $Z_{t}^{n}:=Z_t^{b,n}+Z_t^{\sigma,n}$ and observe that 
    \begin{equation}\label{eq:Yinitexplicit}
    \begin{aligned}
        Y_{t}^{n}\stackrel{d}{=}\mathcal{N}\left(b(x_{n})\int_0^t K(s)\, \mathrm{d}s,\ \sigma^{2}(x_{n}) \int_0^t K(s)^2\mathrm{d}s \right).
    \end{aligned}
    \end{equation}
    Fix a vector $(y_{1},\dots,y_{N})^{\intercal}\in\mathbb{R}^{N}$ which is a continuity point of the distribution function of $\big(Y_{t_{1}}^{\sigma,\infty},\dots,Y_{t_{N}}^{\sigma,\infty}\big)^{\intercal}$. From the definitions of $Y^{n}$, $Z^{n}$ and by introducing $Z_{n}^{N,*}:=\max_{i\in\{1,\dots,N\}}\big|Z_{t_{i}/n}^{n}\big|$, we obtain for $z:\R_+\longrightarrow\R_+$ to be specified below
    \begin{equation}\label{eq:PPP}
    \begin{split}
             &\mathbb{P}\bigg[ \bigcap_{j=1}^{N}\Big\{ \sqrt{\lambda(n)}\big(X_{t_{j}/n}^{x_n}-x_{n}\big)\le y_{j} \Big\} \bigg]            \\ 
            &=\ \mathbb{P}\bigg[ \bigcap_{j=1}^N \Big\{\sqrt{\lambda(n)}\big(Y_{t_{j}/n}^{n}+Z_{t_{j}/n}^{n}\big)\le y_{j} \Big\} \cap \big\{ Z_{n}^{N,*}\le z(t_{N}/n) \big\} \bigg]
            \\ &\hspace{0.55cm}+  \mathbb{P}\bigg[ \bigcap_{j=1}^N \Big\{ \sqrt{\lambda(n)}\big(Y_{t_{j}/n}^{n}+Z_{t_{j}/n}^{n}\big)\le y_{j} \Big\} \cap \big\{ Z_{n}^{N,*} > z(t_{N}/n) \big\} \bigg].
    \end{split}
    \end{equation}
    Note that the moment bounds obtained in Proposition \ref{corollary:zasymptotics} as well as $\sqrt{\mathrm{Var}[Y_{t}^{n}]}$ depend on the starting value only through the continuous, asymptotically non-vanishing factors $(1+|x_{n}|^{p})$ and $\sigma(x_{n})$ (see \eqref{eq:Yinitexplicit}), respectively. By condition~(i) in combination with $\chi_{b}>0$, the interval in~\eqref{eq:zexponentcorridor} is non-degenerate and independent of $n$. Hence, Proposition~\ref{corollary:zasymptotics} also holds for $K\ast\mathrm{d}B$ and $Z$ replaced with $Y^{\sigma, n}=\sigma(x_{n})(K\ast\mathrm{d}B)$ and $Z^{n}$, respectively, with the power function~$z$ being independent of~$n$. Therefore, the last probability in~\eqref{eq:PPP} vanishes as $n\rightarrow\infty$, since we can estimate for $p\ge\max\{2/\chi_{b},2/\chi_{\sigma}\}$:   
    \begin{align}\label{eq:MarkovineqforMaxCLTinitvarying}
        \notag \mathbb{P}\big[Z_{n}^{N,*} > z(t_{N}/n)\big] &\le\frac{\mathbb{E}\big[\big(Z_{n}^{N,*}\big)^{p}\big]}{z(t_{N}/n)^{p}} \le\frac{\sum_{i=1}^{N}\mathbb{E}\big[\big|Z_{t_{i}/n}^{n}\big|^{p}\big]}{z(t_{N}/n)^{p}}\\
        &\lesssim N(1+|x_{n}|^{p})\left(\tfrac{t_{N}}{n}\right)^{\left(\min\left\{\frac{1}{2}+\gamma\hspace{0.02cm}(1+\chi_{b}),\,\gamma \hspace{0.02cm}(1+\chi_{\sigma})\right\}-q\right)\,p},
    \end{align}
    where the exponent is positive by the choice of $q$ according to Proposition \ref{corollary:zasymptotics}.
    
    For the first term in~\eqref{eq:PPP}, we study the asymptotic behaviour of certain upper and lower bounds. We may find an upper bound by estimating
    \begin{align}\label{eq:upperestimateCLTinitvarying}
        \notag &\mathbb{P}\bigg[ \bigcap_{j=1}^N \Big\{\sqrt{\lambda(n)}\big(Y_{t_{j}/n}^{n}+Z_{t_{j}/n}^{n}\big)\le y_{j} \Big\} \cap \big\{ Z_{n}^{N,*}\le z(t_{N}/n) \big\}\bigg]
        \\  &\le \mathbb{P}\bigg[ \bigcap_{j=1}^N \Big\{\sqrt{\lambda(n)}\big(Y_{t_{j}/n}^{n}-z(t_{N}/n)\big)\le y_{j} \Big\}\bigg]
    \end{align}
    which is, for every $n\in\mathbb{N}$, the distribution function of a random vector following a multivariate normal distribution evaluated at $(y_{1},\dots,y_{N})^{\intercal}$. Therefore, by Lévy's continuity theorem, it suffices to study the asymptotic of its mean and covariance structure. For the mean we obtain for every $j\in\{1,\dots,N\}$:
    \begin{equation}\label{eq:CLTinitvarymeanasymptotics}
    \begin{aligned}
        \sqrt{\lambda(n)}\,\big|Y_{t_{j}/n}^{b,n}-z(t_{N}/n)\big|&\lesssim n^{\gamma_{*}}\,\bigg(b(x_{n})\int_{0}^{t_{j}/n}|K(s)|\,\mathrm{d}s+n^{-q}\bigg)
        \\ &\lesssim\Big(b(x_{n})\,n^{\gamma_{*}-\gamma-\tfrac{1}{2}}+n^{\gamma_{*}-q}\Big) \longrightarrow 0,
    \end{aligned}
    \end{equation}
    as $n\rightarrow\infty$, which is an immediate consequence of \eqref{eq:kernelL2order} in combination with the Hölder inequality, condition~(i) as well
    as Proposition~\ref{corollary:zasymptotics}, utilizing $\lambda(n)\sim\big(\mathrm{Var}[Y_{1/n}^{1,n}]\big)^{-1}$. On the other hand, we observe for the covariance matrix for every $i,j\in\{1,\dots,N\}$ with $i\le j$:
    \begin{equation}\label{eq:CLTinitvarycovarianceasymptotics}
    \begin{aligned}
     &\mathrm{cov}\Big( \sqrt{\lambda(n)}\hspace{0.03cm} Y_{t_{i}/n}^{\sigma,n}, \sqrt{\lambda(n)}\hspace{0.03cm} Y_{t_{j}/n}^{\sigma,n}\Big)
    \\ &= \sigma(x_n)^2 \lambda(n) \int_0^{t_{i}/n} K(t_{j}/n - s)K(t_{i}/n-s)\,\mathrm{d}s
    \\ &=  \sigma(x_n)^2 \lambda(n) \int_0^{t_{i}/n} K((t_{j}-t_{i})/n+s) K(s)\,\mathrm{d}s
    \\ &\longrightarrow\sigma(\overline{x})^2  \int_0^{t_{i}} \overline{K}(t_{j}-t_{i} + s)\overline{K}(s)\,\mathrm{d}s,
\end{aligned}
\end{equation}
    as $n\rightarrow\infty$, where we applied condition~(ii). Hence, combining \eqref{eq:CLTinitvarymeanasymptotics}, \eqref{eq:CLTinitvarycovarianceasymptotics} and Lévy's continuity theorem shows that the above random vector converges weakly and its limiting distribution agrees with the law of $(Y^{\sigma,\infty}_{t_{1}},\dots,Y^{\sigma,\infty}_{t_{N}})^{\intercal}$. Since $(y_{1},\dots,y_{N})^{\intercal}$ has been chosen as a continuity point of its distribution function, we obtain from the Portmanteau lemma the following convergence for the upper bound found in~\eqref{eq:upperestimateCLTinitvarying}:
    \begin{equation}\label{eq:upperestimateCLTinitvarying2}
     \lim_{n\rightarrow\infty}\mathbb{P}\bigg[ \bigcap_{j=1}^N \left\{\sqrt{\lambda(n)}\big(Y_{t_{j}/n}^{n}-z(t_{N}/n)\big)\le y_{j} \right\} \bigg]
     = \mathbb{P}\bigg[ \bigcap_{j=1}^{N} \left\{Y^{\sigma,\infty}_{t_j}\le y_{j} \right\} \bigg].
    \end{equation}
    On the other hand, a lower bound may be obtained by estimating
    \begin{align*}
         \mathbb{P}\bigg[ \bigcap_{j=1}^N & \Big\{\sqrt{\lambda(n)}\big(Y_{t_{j}/n}^{n}+Z_{t_{j}/n}^{n}\big)\le y_{j} \Big\} \cap \big\{  Z_{n}^{N,*}\le z(t_{N}/n) \big\} \bigg]
        \\ \ge\mathbb{P}&\bigg[ \bigcap_{j=1}^N \Big\{\sqrt{\lambda(n)}\big(Y_{t_{j}/n}^{n}+z(t_{N}/n)\big)\le y_{j} \Big\} \cap \left\{ Z_{n}^{N,*}\le z(t_{N}/n) \right\} \bigg]
        \\ \ge\mathbb{P}&\bigg[ \bigcap_{j=1}^N \left\{\sqrt{\lambda(n)}\big(Y_{t_{j}/n}^{n}+z(t_{N}/n)\big)\le y_{j} \right\} \bigg] -\mathbb{P}\big[Z_{n}^{N,*} > z(t_{N}/n)\big].
    \end{align*}
    According to \eqref{eq:MarkovineqforMaxCLTinitvarying}, the latter probability vanishes again as $n\rightarrow\infty$. Moreover, by performing the same arguments as for the upper bound, 
    \begin{equation*}
    \lim_{n\rightarrow\infty}\mathbb{P}\bigg[ \bigcap_{j=1}^N \left\{\sqrt{\lambda(n)}\big(Y_{t_{j}/n}^{n}+z(t_{N}/n)\big)\le y_{j}\right\} \bigg]
        =\mathbb{P}\bigg[ \bigcap_{j=1}^N \Big\{Y^{\sigma,\infty}_{t_j}\le y_{j} \Big\} \bigg]
    \end{equation*}
    can be shown since the only difference  with regards to~\eqref{eq:upperestimateCLTinitvarying2} is the sign of $\sqrt{\lambda(n)}\,z(t_{N}/n)$. However, this term asymptotically vanishes in \eqref{eq:CLTinitvarymeanasymptotics}. Hence, applying the multi-dimensional version of the Helly-Bray theorem concludes the proof, as upper and lower bound have the same limit.
\end{proof}

Note that $x_{n}=x_{0}$, $n\in\mathbb{N}$, is, in particular, an admissible choice -- provided that $\sigma(x_{0}) \neq 0$ holds -- in order to obtain a non-degenerate limit distribution. As the next step, we extend this CLT on the finite-dimensional distributions towards a functional CLT that captures the convergence of the process on the path space on any finite interval $[0,T]$.

\begin{corollary}\label{corollary:functionalCLT}
    In the framework of Theorem \ref{theorem:varyinginitialCLT}, assume additionally that $K$ satisfies~\eqref{eq:kernelincrementL2} and that $\gamma':=\gamma\wedge\overline{\gamma}=\gamma_{*}$ holds. Then also the corresponding functional CLT holds, i.e.\ for each $T > 0$ we have as $n\rightarrow\infty$:
    \begin{equation}\label{eq:functionalCLT}
        \left(\sqrt{\lambda(n)}\hspace{0.02cm}\big(X_{t/n}^{x_{n}}-x_{n}\big)\right)_{t\in [0,T]}\stackrel{d}{\longrightarrow} (Y^{\sigma, \infty})_{t\in [0,T]}.
    \end{equation}
\end{corollary}
\begin{proof}
    The weak convergence of the finite-dimensional distributions of the processes on $[0,T]$ given as in \eqref{eq:functionalCLT} is an immediate consequence of Theorem \ref{theorem:varyinginitialCLT}, where we implicitly assume $n > T$ so that $t/n < 1$ holds for all $t\in [0,T]$. Hence, it remains to show the tightness of the laws with respect to the uniform topology. To this end, by an application of the moment estimate obtained in \cite[Lemma 2.4]{AbiJaLaPu19} for $p\ge 2$, we obtain for arbitrary $t, t'\in [0,T]$ and $n > T$:
    \begin{equation*}\label{eq:functionalCLT1}
        \mathbb{E}\big[\big|X_{t/n}^{x_{n}}-X_{t'/n}^{x_{n}}\big|^{p}\big]\le C\sup_{u\in [0,T]}\mathbb{E}\big[|b(X_{u}^{x_{n}})|^{p}+|\sigma(X_{u}^{x_{n}})|^{p}\big]\,n^{-\gamma' p}\,(t-t')^{\gamma' p},
    \end{equation*}
    and by linear growth combined with the moment estimate \eqref{eq:momentestimatewithg} from Lemma \ref{lemma:momentestimatewithg} for $g_n\equiv x_n$ as well as $(x_n)_{n\in\mathbb{N}}$ being convergent we can conclude
    \begin{displaymath}
       \sup_{n\in\mathbb{N}, n > T}\sup_{u\in [0,T]}\mathbb{E}\big[|b(X_{u}^{x_{n}})|^{p}+|\sigma(X_{u}^{x_{n}})|^{p}\big]<\infty. 
    \end{displaymath}
    Hence, due to $\gamma'=\gamma_*$ combined with the bound $\lambda(n) \lesssim n^{2\gamma_*}$ which follows from \eqref{eq: lambda definition} and \eqref{eq:kernelL2order}, we arrive at
    \begin{equation*}\label{eq:functionalCLT2}
    \begin{aligned}
        &\mathbb{E}\left[\left|\sqrt{\lambda(n)}\big(X_{t/n}^{x_{n}}-x_{n}\big)-\sqrt{\lambda(n)}\big(X_{t'/n}^{x_{n}}-x_{n}\big)\right|^{p}\right] 
        \\ &\qquad \qquad \lesssim n^{\gamma_{*} p}\,\mathbb{E}\big[\big|X_{t/n}^{x_{n}}-X_{t'/n}^{x_{n}}\big|^{p}\big] \lesssim (t-t')^{\gamma' p},
    \end{aligned}
    \end{equation*}
    which, combined with $X_{0/n}^{x_n}-x_{n}=0$ for all $n\in\mathbb{N}$, gives Kolmogorov's tightness criterion (see e.g.\ \cite[Theorem XIII.1.8]{revuzyor}) for $p>(1/\gamma'\vee 2)$ and hence completes the proof. 
\end{proof}

This result is applicable for the Riemann-Liouville kernel
(see Example~\ref{example:limitingkernelfractional} below) with $H \in (0,1/2]$, but not for $H > 1/2$, since then we cannot verify the tightness condition, as we generally obtain $\gamma=\gamma_*=H$ and $\overline{\gamma}=\min\{H,1/2\}$ (see also Example \ref{example:limitingkernelfractional}). Let us further remark that the CLT and fCLT obtained in this section could be also shown for time-dependent coefficients $b, \sigma:\R_{+}\times\R\rightarrow\R$, provided that our main assumption on the H\"older continuity holds with a constant independent of the time variable.

\subsection{Extension by the delta method}\label{subsection:SVIECLTdeltamethod}

Another possible extension concerns the CLT for the transformed process $f(X^{x_{n}})$, where $f:\mathbb{R}\rightarrow\mathbb{R}$ is sufficiently smooth. In such a case we need to study the family of the transformed normalized processes on $\R_+$ given by
\begin{align}\label{eq:fsequence}
    \sqrt{\lambda(n)}\left( f\big(X^{x_n}_{\cdot/n}\big) - f(x_n)\right), \quad n \in \mathbb{N}.
\end{align}
By using the so-called delta method, we obtain the following result
for the finite-dimensional distributions: 
\begin{corollary}\label{corollary:c2trafosfinitedimCLT}
    Under the same assumptions as in Theorem \ref{theorem:varyinginitialCLT}, assume that for each $n \geq 1$ there exists a continuous solution $X^{x_n}$ of \eqref{eq:generalSVIE},
    and let $f\in C^{1}(\mathbb{R};\mathbb{R})$. Then
    \begin{equation}\label{eq:C2CLT}
    \begin{aligned}
        &\left(\sqrt{\lambda(n)}\big(f\big(X_{t_{j}/n}^{x_{n}}\big)-f(x_{n})\big) \right)_{j=1,\dots, N}
        \stackrel{d}{\longrightarrow} &f'(\overline{x})\left(Y_{t_{j}}^{\sigma,\infty} \right)_{j=1,\dots,N}.
    \end{aligned}
    \end{equation}
    Moreover, if $f\in C^{2}(\mathbb{R};\mathbb{R})$, $f'(\overline{x})=0$ and $\lim_{n\rightarrow\infty}\sqrt{\lambda(n)}f'(x_{n})=0$, then 
    \begin{equation}\label{eq:C2CLT2ndorder}
    \begin{aligned}
        &\left(\lambda(n)\big(f\big(X_{t_{j}/n}^{x_{n}}\big)-f(x_{n})\big) \right)_{j=1,\dots, N} \stackrel{d}{\longrightarrow} &\frac{f''(\overline{x})}{2}\left(\big(Y_{t_{j}}^{\sigma,\infty}\big)^{2} \right)_{j=1,\dots,N},
    \end{aligned}
    \end{equation}
    and so the one-dimensional marginals follow a scaled $\chi^{2}$-distribution with one degree of freedom.
\end{corollary}
\begin{proof}
    Firstly, let us note that, by assumption, the solutions $X^{x_n}$ are constructed on filtered probability spaces $(\Omega_n, \mathcal{F}_n, \mathbb{F}_n, \mathbb{P}_n)$. To apply Slutsky's theorem, we need to find a realization on a joint probability space $(\Omega, \mathcal{F}, \mathbb{F}, \mathbb{P})$. The latter is always possible as we may, e.g., consider the product space on which all $X^{x_n}$ are independent. From now on, let us work with such a realization. Moreover, without loss of generality we may assume that $T < 1$, as we may always consider $n > T$.
    
    First, performing for every $i\in\{1,\dots,N\}$ and $n \geq 1$ a first-order Taylor expansion leads to
    \begin{equation}\label{eq:C2CLTTaylor} \sqrt{\lambda(n)}\big(f\big(X_{t_{i}/n}^{x_{n}}\big)-f(x_{n})\big)= \sqrt{\lambda(n)}f'(\xi_{i,n})\big(X_{t_{i}/n}^{x_{n}}-x_{n}\big),
    \end{equation}
    where the random variable $\xi_{i,n}$ satisfies $|\xi_{i,n}-x_{n}|\le \big|X_{t_{i}/n}^{x_{n}}-x_{n}\big|$ pointwise. Hence, we can conclude $\lim_{n\rightarrow\infty}f'(\xi_{i,n})=f'(\overline{x})$ in $\mathbb{P}$-probability for every $i\in\{1,\dots,N\}$ by
    \begin{displaymath}
        \lim_{n\rightarrow\infty}\mathbb{E}\left[\big|X_{t_{i}/n}^{x_{n}}-x_{n}\big|^{p}\right]=0,
    \end{displaymath}
    following from \eqref{eq:momentestimatewithg} for $p\geq2$ in combination with $g_n\equiv x_n$ and $(x_n)_{n\in\mathbb{N}}$ being convergent, Vitali's convergence theorem, the continuous mapping theorem as well as the continuity of $f'$. Finally, using representation \eqref{eq:C2CLTTaylor}, the desired weak convergence \eqref{eq:C2CLT} is an immediate consequence of Theorem~\ref{theorem:varyinginitialCLT} and an application of Slutsky's theorem.
    
    For the case $f'(\overline{x})=0$, where $\lim_{n\rightarrow\infty}\sqrt{\lambda(n)}f'(x_{n})=0$, we use a similar argument but now perform a second order Taylor expansion. Hence, by adjusting the order of the normalizing sequence, the analogue of \eqref{eq:C2CLTTaylor} becomes in this case
    \begin{equation}\label{eq:C2CLTTaylor2ndorder}
    \begin{aligned} \lambda(n)\big(f\big(X_{t_{i}/n}^{x_{n}}\big)-f(x_{n})\big)=\, &\lambda(n)\hspace{0.03cm}f'(x_{n})\big(X_{t_{i}/n}^{x_{n}}-x_{n}\big)\\
        &+ \lambda(n)\hspace{0.03cm}\frac{f''(\xi_{i,n})}{2}\big(X_{t_{i}/n}^{x_{n}}-x_{n}\big)^{2},
    \end{aligned}
    \end{equation}
    where $|\xi_{i,n}-x_{n}|\le \big|X_{t_{i}/n}^{x_{n}}-x_{n}\big|$ holds again pointwise. Combining \eqref{eq:varyinginitialCLT} from Theorem \ref{theorem:varyinginitialCLT} with $\lim_{n\rightarrow\infty}\sqrt{\lambda(n)}f'(x_{n})=0$ and Slutsky's theorem shows that the vector of the first summands on the right-hand side of \eqref{eq:C2CLTTaylor2ndorder} converges weakly towards the zero vector and, therefore, also in $\mathbb{P}$-probability. Moreover, for the remainder, we can conclude from the continuous mapping theorem in combination with again \eqref{eq:varyinginitialCLT}, $\lim_{n\rightarrow\infty}f''(\xi_{i,n})=f''(\overline{x})$ in $\mathbb{P}$-probability, following analogously to above, as well as Slutsky's theorem that it converges in distribution to a random vector of the form
    \begin{equation*} 
    \begin{aligned}
    &\left(\tfrac{f''(\overline{x})}{2} (Y_{t_{j}}^{\sigma,\infty})^{2} \right)_{j=1,\dots,N}
    \stackrel{d}{=}\,&\left( \tfrac{f''(\overline{x})}{2}\,\big\|\overline{K}\big\|_{L^{2}((0,t_{j}])}^{2} \, \sigma(\overline{x})^{2} \chi_{1,j}^{2} \right)_{j=1,\dots,N},
    \end{aligned}
    \end{equation*}
    where each marginal $\chi_{1,j}^{2}$ is $\chi^{2}$-distributed with one degree of freedom for every $j\in\{1,\dots,N\}$, and the dependence structure is inherited from $Y^{\sigma,\infty}$.
\end{proof}

\begin{remark}
    It is easy to see that the tightness argument given in Corollary~\ref{corollary:functionalCLT} also can be applied to~\eqref{eq:fsequence} for Lipschitz continuous $C^{1}$-transform\-ations of $X^{x_{n}}$. Thus, in combination with Corollary \ref{corollary:c2trafosfinitedimCLT}, we obtain a functional generalization of \eqref{eq:C2CLT} of the form
    \begin{displaymath} \left(\sqrt{\lambda(n)}\hspace{0.03cm}\big(f\big(X_{t/n}^{x_{n}}\big)-f(x_{n})\big)\right)_{t\in [0,T]}\stackrel{d}{\longrightarrow} f'(\overline{x})\hspace{0.03cm}\big(Y_{t}^{\sigma, \infty}\big)_{t\in [0,T]},\quad \mbox{as}\ \, n\rightarrow\infty.
    \end{displaymath}
\end{remark}

Finally, let us briefly discuss an extension where $f$ fulfills the required smoothness only locally at $x_{0}$, i.e.\ there exists $\varepsilon>0$ such that $f|_{B(x_{0},\varepsilon)}\in C^{i}\big(B(x_{0},\varepsilon);\mathbb{R}\big)$, where $i\in\{1,2\}$ and $B(x_{0},\varepsilon)$ denotes the open ball around $x_{0}\in\mathbb{R}$ with radius $\varepsilon$. It is then straightforward to prove \eqref{eq:C2CLT} and \eqref{eq:C2CLT2ndorder} also in this case for $x_{n}\equiv x_{0}$. First, we notice that \eqref{eq:C2CLT} is equivalent to the corresponding result, if one replaces each $X^{x_{0}}$ with an appropriately stopped version, i.e.
    \begin{equation}\label{eq:C2CLT2}
        \begin{aligned} \left(\sqrt{\lambda(n)}\hspace{0.03cm}\big(f\big(X_{(t_{j}/n)\wedge\tau}^{x_{0}}\big)-f(x_{0})\big) \right)_{j=1,\dots,N} \stackrel{d}{\rightarrow} f'(x_0)\hspace{0.03cm}\big( Y_{t_{j}}^{\sigma,\infty} \big)_{j=1,\dots,N},
    \end{aligned}
    \end{equation}
    where $\tau$ is a weak stopping time which is almost surely strictly positive. Indeed, this is an immediate consequence of the triangle inequality as well as
    \begin{align*} \lim_{n\rightarrow\infty}\bigg|&\mathbb{E}\Big[g\Big(\sqrt{\lambda(n)}\hspace{0.03cm}\big(f\big(X_{t_{1}/n}^{x_{0}}\big)-f(x_{0})\big),\dots,\sqrt{\lambda(n)}\hspace{0.03cm}\big(f\big(X_{t_{N}/n}^{x_{0}}\big)-f(x_{0})\big)\Big)\Big] \\
    &-\mathbb{E}\Big[g\Big(\sqrt{\lambda(n)}\hspace{0.03cm}\big(f\big(X_{(t_{1}/n)\wedge\tau}^{x_{0}}\big)-f(x_{0})\big),\dots,\sqrt{\lambda(n)}\hspace{0.03cm}\big(f\big(X_{(t_{N}/n)\wedge\tau}^{x_{0}} \big)-f(x_{0})\big)\Big)\Big]\bigg|=0,
    \end{align*}
    as $n\rightarrow\infty$ for every $g\in C_{b}(\mathbb{R}^{N};\mathbb{R})$ by the dominated convergence theorem. In particular, 
    \begin{displaymath}
        \tau_{\varepsilon}^{0}:=\inf\{t\in\R_+:\hspace{0.1cm} |X_{t}^{x_{0}}-x_{0}|\ge \varepsilon/2\big\}\wedge T
    \end{displaymath}
    is an admissible choice. As the associated stopped process is restricted to $B(x_{0},\varepsilon)$, we can perform a Taylor expansion and complete the proof analogously to above utilizing the continuity of the paths of $X_{\cdot\wedge\tau_{\varepsilon}^{0}}^{x_{0}}$ as well as \eqref{eq:C2CLT2} for $f=\id$ according to Theorem \ref{theorem:varyinginitialCLT} and the above equivalence. This also applies to the second part of the corollary, i.e.~\eqref{eq:C2CLT2ndorder}. 
    
    Note that this localization method may also be used for the case $x_{n}\not\equiv x_{0}$, if~$f$ has the required smoothness on $B(\overline{x},\varepsilon)$ and the additional condition 
    \[
        \inf_{n \geq \tilde{n}} \tau_{\varepsilon}^{n}>0 \ \ \text{a.s.},
    \]
    holds, where $\tilde{n}:= \min\{n\in\mathbb{N}:\hspace{0.05cm} |x_{m}-\overline{x}|<\varepsilon/2, \,\forall m\ge n\}$. Note that while this random time might, in general, not be a stopping time, it is still a weak stopping time. A crucial step for the argument is the moment estimate
    \begin{equation*}
        \mathbb{E}\left[\big|X_{(t_{i}/n)\wedge\tau_{\varepsilon}^{n}}^{x_{n}}-x_{n}\big|^{p}\right]\le \mathbb{E}\left[\big|X_{t_{i}/n}^{x_{n}}-x_{n}\big|^{p}\right] +\left(\frac{\varepsilon}{2}\right)^{p}\,\mathbb{P}[t_{i}/n>\tau_{\varepsilon}^{n}]
    \end{equation*}
    converging to $0$ as $n\rightarrow\infty$ according to \eqref{eq:momentestimatewithg} for $g_n\equiv x_n$ with $(x_n)_{n\in\mathbb{N}}$ being convergent as well as the dominated convergence theorem in combination with $\inf_{n \ge \tilde{n}}\tau_{\varepsilon}^{n}>0$ a.s.

\subsection{Examples of kernels and limiting kernels}\label{subsection:CLTKernelExamples}

We conclude this section with a few examples of Volterra kernels covered by our results presented in Theorem \ref{theorem:varyinginitialCLT} and the subsequent extensions. Firstly, the classical Riemann-Liouville kernel satisfies the assumptions of our theorem as illustrated in the next example.

\begin{example}\label{example:limitingkernelfractional}
    The Riemann-Liouville kernel $K(t) = \frac{t^{H - 1/2}}{\Gamma(H+1/2)}$, $t\in\R_+$, gives $H = \gamma = \gamma_*$, the scaling is given by
        \[
            \lambda(n) = \left( \int_0^{1/n}K(r)^2\,\mathrm{d}r\right)^{-1} = 2H \,\Gamma(H + 1/2)^2 \,n^{2H},
        \]
        and the limit kernel for the covariances in condition (ii) of Theorem \ref{theorem:varyinginitialCLT} exists and is given by 
        \begin{align}\label{eq: overline K}
            \overline{K}(t) = \sqrt{2H}\hspace{0.03cm}t^{H - 1/2}, \quad t\in\R_+,
        \end{align}
        as can be seen from the substitution $r \to r/n$ combined with the scaling property of the kernel. Moreover, condition \eqref{eq:kernelincrementL2} holds with $\overline{\gamma}=\min\{H, 1/2\}$ by \cite[Example 2.3]{AbiJaLaPu19}, whence the functional CLT from Corollary \ref{corollary:functionalCLT} holds for $0<H\le 1/2$.
\end{example}

The next example illustrates that only the asymptotics of the kernel as $t \to 0$ plays a role with regards to conditions (i) and (ii) in Theorem \ref{theorem:varyinginitialCLT}. 

\begin{example}\label{example:limitingkernelmorethanRL}
    Suppose that there exists $C(H) \in \R_+^*$ such that
    \[
       K(t) \sim C(H)\hspace{0.03cm}t^{H-1/2}, \quad \  \mbox{as} \ t \to 0.
    \]
    Then $\gamma = \gamma_* = H$, $\lambda(n) \sim 2H\hspace{0.03cm}C(H)^{-2}\hspace{0.03cm}n^{2H}$, and there exists a constant $C'(H) \in \R_+^*$ such that $\overline{K}(t) = C'(H)\hspace{0.03cm}t^{H - 1/2}$, by the dominated convergence theorem. Therefore, conditions~(i) and~(ii) in Theorem \ref{theorem:varyinginitialCLT} are satisfied. Note that this example covers also regular kernels $K \in C^1(\R_+)$ with $K(0) \in \R_+^*$ by taking $H = 1/2$. In particular, for all such regular kernels, we obtain $\overline{K}\equiv 1$ and, hence, the limit process is even a time-homogeneous Markov process, which, in general, contrasts the original process. An important subclass of the above kernels has representation $K(t)=l(t)\hspace{0.03cm}t^{H-1/2}$, where $l:\R_+\rightarrow \R$ is locally Lipschitz. These kernels also satisfy \eqref{eq:kernelincrementL2} with $\overline{\gamma}=\min\{H,1/2\}$ (see \cite[Example 2.3~(iv)]{AbiJaLaPu19}), whence for $H\le 1/2$ also the functional CLT from Corollary~\ref{corollary:functionalCLT} is applicable.
\end{example}

The above example covers, in particular, the gamma kernel, the sum of exponentials and regularized fractional kernels, e.g.\ via a time shift. Below we close this section with an example of a $\log$-modulated Riemann-Liouville kernel.

\begin{example}\label{example:kernelssatisfying CLT}
    For the $\log$-modulated fractional kernel 
    \[
            K(t) = \frac{t^{H - 1/2}}{\Gamma(H+1/2)}\log(1 + 1/t), \quad t\in\R_+,
    \]
    \eqref{eq:kernelL2order} holds for $\gamma_*=H$ and all $\gamma \in (0,H)$. Hence, condition (i) of Theorem \ref{theorem:varyinginitialCLT} is satisfied as $\chi_\sigma>0$. Moreover, the scaling satisfies $\lambda(n) \sim \frac{2H\hspace{0.03cm}\Gamma(H+1/2)^2\hspace{0.03cm} n^{2H}}{\log^2 (n)}$, as can be seen either by explicit integration or an application of Karamata's theorem~\cite[Proposition~1.5.10]{BiGoTe87}:
        \[
          \int_0^{1/n} K(r)^2 \,\mathrm{d}r = \int_n^\infty \frac{u^{-1-2H}}{\Gamma(H+1/2)^2}\log^2(1+u) \,\mathrm{d}u
          \sim \frac{1}{2H\hspace{0.03cm}\Gamma(H+1/2)^2} n^{-2H}\log^2 (n).
        \] 
        A short computation shows that also condition~(ii) in Theorem~\ref{theorem:varyinginitialCLT} holds with $\overline{K}$ given by \eqref{eq: overline K}.
        However, as $\gamma < H = \gamma_*$, we cannot infer a functional CLT via Corollary \ref{corollary:functionalCLT}. 
\end{example}

Finally, let us remark that the limiting kernel $\overline{K}$, provided it exists, satisfies the identity 
\begin{align}
    \int_0^t \overline{K}(s)^2 \mathrm{d}s &= \lim_{n \to \infty} \lambda(n) \int_0^{t/n} K(s)^2 \,\mathrm{d}s \notag
    \\ &= \lim_{n \to \infty} \frac{\int_0^{t/n}K(s)^2 \,\mathrm{d}s}{\int_0^{1/n}K(s)^2 \,\mathrm{d}s}
    = \lim_{n \to \infty} \frac{t K(t/n)^2 }{K(1/n)^2}. \label{eq:find bar K}
\end{align}
The latter could be used as an alternative to find $\overline{K}$ without the use of integrals.
For instance, if~$K$ varies regularly at zero, i.e.\ $K(t)=\ell(t)\hspace{0.03cm}t^{H-1/2}$
with~$\ell$ a slowly varying function, then~\eqref{eq:find bar K} readily
implies $\overline{K}(t)=\sqrt{2H}\hspace{0.03cm} t^{H-1/2}$.

\subsection{Implications for mathematical finance}\label{subsection:CLTmathfinance}

The large deviations results mentioned in the introduction aim at asymptotics for out-of-the-money (OTM) vanilla or digital options in rough
volatility models.
For \emph{at-the-money} (ATM) vanilla options, very general
results, way beyond rough volatility, have
been proven~\cite{AlLeVi07,Fu17}.
Unlike OTM, ATM digital calls can have
a different asymptotic behavior than
ATM vanilla calls, which is amenable to CLTs.
Indeed, if $\sigma(\overline{x}) \neq 0$ in Theorem~\ref{theorem:varyinginitialCLT}, then we obtain
\begin{displaymath}
    \lim_{n\to\infty}\mathbb{E}[\mathbbm{1}_{\{X_{T/n}\ge x_{0}\}}]
    =\lim_{n\to\infty}\mathbb{P}\left[\sqrt{\lambda(n)}\,\big(X_{T/n}-x_{0}\big)\ge 0\right] =\frac{1}{2},
\end{displaymath}
and if~$X$ models a financial asset, this could be read as a statement about the price
of ATM digital call options with maturities $T/n$
(cf.\ Section~4 in~\cite{GeKlPoSh15}).
However, note that SVIE solutions are in general not semimartingales.
We refer to~\cite{BeSoVa08,Gu06} and the references therein for no-arbitrage theory for non-semimartingale
models. Still, as no convenient option pricing theory is available so far,
the common way to use SVIEs in financial modeling is via stochastic
volatility models (see~\cite{rvol}). Consider a price process~$S$ satisfying $\mathrm{d}S_t/S_t=\sqrt{v_t}\,\mathrm{d}\tilde{B}_t$,
where~$v\geq0$ solves~\eqref{eq:generalSVIE}, and~$\tilde B$ is a Brownian motion
correlated with~$B$. As long as we assume that~$S$ is a martingale under a risk-neutral measure, it does not matter whether the variance process~$v$, which is not a tradable asset, is a semimartingale or not.

To invoke our functional CLT obtained in Corollary \ref{corollary:functionalCLT}, we
shift attention to the small-time behaviour of volatility derivatives, in particular options on the realized variance~$v$ (see~\cite{LaMuSt21}). 
The payoff of such a contract 
at time $T>0$ is a function of the average variance over $[0,T]$, i.e.\ of 
\begin{equation}\label{eq:integratedvariance}
    \mathcal{V}_{T}:=\frac{1}{T}\int_{0}^{T}v_{t}\,\mathrm{d}t.
\end{equation}
First, we study the small-time limit of an ATM digital call with underlying~$\mathcal{V}$,
assuming zero interest rate. By Lebesgue's differentiation theorem, we have $\mathcal{V}_{0}:=\lim_{T\searrow 0}\mathcal{V}_{T}=v_{0}$ a.s. For maturity~$1/n$, our claim has the payoff $\mathbbm{1}_{\{\mathcal{V}_{1/n}\ge v_{0}\}}$, which implies for its price as $n\rightarrow\infty$:
\begin{equation}\label{eq:volatilityderivativesATM}
\begin{aligned}
    \mathbb{E}\big[\mathbbm{1}_{\{\mathcal{V}_{1/n}\ge v_{0}\}}\big]&=\mathbb{P}\left[\sqrt{\lambda(n)}\,\big(\mathcal{V}_{1/n}-v_{0}\big)\ge 0\right]\\
    &=\mathbb{P}\left[n\int_{0}^{1/n}\sqrt{\lambda(n)}\,(v_{t}-v_{0})\,\mathrm{d}t\ge 0\right]\\
    &=\mathbb{P}\left[\int_{0}^{1}\sqrt{\lambda(n)}\,(v_{t/n}-v_{0})\,\mathrm{d}t\ge 0\right] \stackrel{n\rightarrow\infty}{\longrightarrow}\mathbb{P}\left[\int_{0}^{1}Y_{t}^{\sigma, \infty}\,\mathrm{d}t\ge 0\right]=\frac{1}{2},
\end{aligned}
\end{equation}
where the convergence follows from the functional CLT from Corollary \ref{corollary:functionalCLT}, the continuous mapping theorem and the integral being a continuous functional w.r.t.\ the uniform topology. Moreover, the latter integral defines a non-degenerate, centered Gaussian random variable by the stochastic Fubini theorem and $\overline{K}\not\equiv 0$, whence the 
limit price is again~$\tfrac{1}{2}$. For the relevance of higher
order terms beyond the limit price, which we do not consider in this paper,
we refer to results on the implied volatility skew in~\cite{GeGuPi16,GeKlPoSh15,JaTo20}.

Moreover, we can investigate the regime called ``almost  ATM'' (AATM)
 as discussed in~\cite{FrGePi18}, and its boundary case. For this purpose, consider again a digital call on $\mathcal{V}_{1/n}$ with  strike $v_{0}+n^{-\beta}a$, where $a\in\R$ and $\beta> 0$. We can compute similarly to above 
\begin{equation}\label{eq:volatilityderivativesprep}
\begin{aligned}
    \mathbb{E}\big[\mathbbm{1}_{\{\mathcal{V}_{1/n}\ge v_{0} + n^{-\beta}a\}}\big]
    &= \mathbb{P}\left[n\int_{0}^{1/n}\,(v_{t}-v_{0})\,\mathrm{d}t\ge n^{-\beta}a\right]
    \\ &= \mathbb{P}\left[\int_{0}^{1}\sqrt{\lambda(n)}\,(v_{t/n}-v_{0})\,\mathrm{d}t\ge \sqrt{\lambda(n)}\hspace{0.03cm} n^{-\beta}a\right].
\end{aligned}
\end{equation}
Depending on $\alpha$ and $\beta$, the limit can now be calculated.
\begin{proposition}\label{proposition:volatilityderivatives}
    Let $v\ge 0$ be a solution to \eqref{eq:generalSVIE} with $\sigma(v_{0})> 0$ in the framework of Theorem~\ref{theorem:varyinginitialCLT}, assuming additionally that $K$ satisfies \eqref{eq:kernelincrementL2} and $\gamma'=\gamma_{*}$. Then the small-time limits of digital call option prices with underlying $\mathcal{V}$ and strike $v_{0}+n^{-\beta}a$, where $a\in\R$ and $\beta> 0$
    are given by:
    \begin{itemize}
        \item[(i)] $a=0$:
        \begin{equation}\label{eq:volatilityderivativesATMprop}
        \lim_{n\rightarrow\infty}\mathbb{E}\left[\mathbbm{1}_{\{\mathcal{V}_{1/n}\ge v_{0}\}}\right]=\frac{1}{2},
        \end{equation}
        \item[(ii)] $a\neq 0$, $\beta>\gamma_*$: 
        \begin{equation}\label{eq:volatilityderivativesAATM}
        \lim_{n\rightarrow\infty}\mathbb{E}\left[\mathbbm{1}_{\{\mathcal{V}_{1/n}\ge v_{0}+n^{-\beta}a\}}\right]=\frac{1}{2},
        \end{equation}
        \item[(iii)] If $a\neq 0$, $\beta=\gamma_*$ and $\sqrt{\lambda(n)}\sim C_{\lambda} n^{\gamma_*}$ for some $C_{\lambda}>0$, then the limit can be expressed
        in terms of the cdf of the standard normal distribution:
        \begin{equation}\label{eq:volatilityderivativeslimitAATM}
        \hspace{-1cm}\lim_{n\rightarrow\infty}\mathbb{E}\left[\mathbbm{1}_{\{\mathcal{V}_{1/n}\ge v_{0}+n^{-\gamma_{*}}a\}}\right]=1-\Phi\left(\frac{\sigma(v_0)^{-1}C_{\lambda}a}{\sqrt{\int_{0}^{1}\big(\int_{0}^{s}\overline{K}(t)\,\mathrm{d}t\big)^2\,\mathrm{d}s}} \right).
        \end{equation}
    \end{itemize}
\end{proposition}
\begin{proof}
     First, notice that the assumptions directly imply that the functional CLT from Corollary \ref{corollary:functionalCLT} holds for $v$. Moreover, we recall $\sqrt{\lambda(n)}\lesssim n^{\gamma_{*}}$.
     
     $(i)$: This corresponds to the ATM case studied above. Therefore, the
     limit price is $1/2$ by~\eqref{eq:volatilityderivativesATM}.
     
     $(ii)$: This case is very similar to the AATM case in \cite{FrGePi18}. Combining~\eqref{eq:volatilityderivativesprep} with $\sqrt{\lambda(n)}\lesssim n^{\gamma_{*}}$, $\beta>\gamma_*$ and Slutsky's theorem shows that the limit price is also $1/2$ here.
     
     $(iii)$: Here the statement follows from $\sqrt{\lambda(n)}\sim C_{\lambda} n^{\gamma_*}$, taking the limit in \eqref{eq:volatilityderivativesprep}, our functional CLT, which is again preserved under the continuous integral operator, and Slutsky's theorem, where $\mathrm{sd}$ denotes the standard deviation:
    \begin{align*} \lim_{n\rightarrow\infty}&\mathbb{P}\left[\int_{0}^{1}\sqrt{\lambda(n)}\,(v_{t/n}-v_{0})\,\mathrm{d}t\ge \sqrt{\lambda(n)}\hspace{0.03cm}n^{-\gamma_*}a\right]\\
    =\ &\mathbb{P}\left[\int_{0}^{1}Y_{t}^{\sigma, \infty}\,\mathrm{d}t\ge C_{\lambda}a\right]=1-\Phi\left(\frac{C_{\lambda}a}{\mathrm{sd}\left[\int_{0}^{1}Y_{t}^{\sigma, \infty}\,\mathrm{d}t\right]}\right).
    \end{align*}
    Finally, we obtain from the stochastic Fubini theorem
    \begin{align*}
        \mathrm{sd}\left[\int_{0}^{1}Y_{t}^{\sigma, \infty}\,\mathrm{d}t\right]&=\sigma(v_0)\,\mathrm{sd}\left[\int_{0}^{1}\int_{s}^{1}\overline{K}(t-s)\,\mathrm{d}t\,\mathrm{d}B_{s}\right]\\ &=\sigma(v_0)\,\sqrt{\int_{0}^{1}\left(\int_{s}^{1}\overline{K}(t-s)\,\mathrm{d}t\right)^2\,\mathrm{d}s}
        = \sigma(v_0)\hspace{0.03cm} \sqrt{\int_0^1 \left( \int_0^{s} \overline{K}(t)\, \mathrm{d}t \right)^2\, \mathrm{d}s },  
    \end{align*}
    which proves \eqref{eq:volatilityderivativeslimitAATM}.
\end{proof}
Part (iii) is the boundary case where the regime switch from AATM into ``moderately out of the money'' (MOTM, see~\cite{FrGePi18}) occurs for $a>0$. Here the limit price is in general not~$1/2$ anymore, as~\eqref{eq:volatilityderivativeslimitAATM} shows. In the following, we investigate this phenomenon for Riemann-Liouville kernels $K_{H}(t)=t^{H-1/2}$, $t\in\R_+$.
\begin{example}\label{example:volatilityyderivativesRL}
 In the Riemann-Liouville case, we have by Example \ref{example:limitingkernelfractional} the limiting kernel $\overline{K}_{H}=\sqrt{2H}K_{H}$ as well as $C_{\lambda}=\sqrt{2H}$, and Corollary~\ref{corollary:functionalCLT} is applicable for $\gamma_{*}=H\in (0,1/2]$. Hence, we obtain from \eqref{eq:volatilityderivativeslimitAATM} the asymptotic price 
    \begin{equation}\label{eq:volatilityyderivativesRL} \lim_{n\rightarrow\infty}\mathbb{E}\big[\mathbbm{1}_{\{\mathcal{V}_{1/n}\ge v_{0} + n^{-H}a\}}\big]=1-\Phi\left(\sqrt{2H+2}\left(H+\tfrac{1}{2}\right)\frac{a}{\sigma(v_0)}\right),
    \end{equation}
    which, depending on $a\in\R$, may attain any value in $(0,1)$. In particular, \eqref{eq:volatilityyderivativesRL} holds in the rough Heston model, see \eqref{eq:roughHeston}. Moreover, it is an immediate consequence of Example~\ref{example:limitingkernelmorethanRL}
    that~\eqref{eq:volatilityyderivativesRL} persists for the gamma kernel $K_{\beta, H}(t)=t^{H-1/2}e^{-\beta t}$, $t\in\R_+$, with $H\in(0,1/2]$ and $\beta>0$, since $\overline{K}_{\beta, H}=\overline{K}_{H}$ and the same holds true for $C_{\lambda}$~and $\gamma_{*}$.
\end{example}
    We now outline a potential application of \eqref{eq:volatilityyderivativesRL}.
    Given a collection of prices of digital calls on $\mathcal{V}$ with sufficiently small time to maturity and strikes close to ATM, which is natural for these maturities (see \cite{FrGePi18}), the above result may be used for calibrating the parameter~$H$. Select $N\in\mathbb{N}$ and choose a collection of points $(H_{i})_{i\in\{1,\dots,N\}}$ in $(0,1/2]$ that may be equidistant. Now given a collection of $M$ digital calls on the average realized variance $\mathcal{V}$ with time to maturities, strikes and prices $(T_{i}, K_{i}, \pi_{i})_{i\in\{1,\dots,M\}}$ with $T_{i}<\delta$ and $|K_{i}-v_{0}|<\Delta$, where $\delta, \Delta>0$ are tuneable hyperparameters, we can analyze the loss function,  specialized to the Riemann-Liouville and gamma case, i.e.
    \begin{equation*}
        L(H)=\sum_{i=1}^{M}\left|1-\Phi\left(\sqrt{2H+2}\left(H+\tfrac{1}{2}\right)\frac{a(H, i)}{\sigma(v_0)}\right)-\pi_i\right|^2, \quad H\in (0,1/2],
    \end{equation*}
    where we defined $a(H,i):=n_{i}^H(K_i-v_0)$ with $n_i:=\lfloor T_{i}^{-1}\rfloor$. An estimator~$\widehat{H}$ may now be obtained by determining
    \begin{equation*}
        \widehat{H}:=\argmin_{H_{i}, \,i\in\{1,\dots,N\}}L(H_{i}).
    \end{equation*}
    Note that $a(H,i)$ should stay bounded for shrinking maturities, in order to apply
    Proposition~\ref{proposition:volatilityderivatives}~(iii), but as~$H$ is small
    in practice and by potentially modeling $\Delta$ also as a decaying function as $T_i\searrow 0$, we hope that the factor $n_i^H$ in the definition of $a(H,i)$
    will not be an obstacle in a numerical implementation.
    The approach is robust in the sense that little information on the dynamics of~$v$ needs to be specified; on the other hand, data availability is an issue, as these digital options are part of the
    OTC (over-the-counter) market, and are thus usually not liquidly traded on exchanges. 
    
    Finally, note that considering strikes of the form $v_{0}+n^{-\beta}a$ with $a>0$ and $\beta\in[0,\gamma_{*})$ leads to conceptually very different large and moderate deviations
    regimes; see~\cite{CaPa23,LaMuSt21} and the references therein.

\section{Hilbert space valued Markovian lifts of SVIEs}\label{section:Liftpreliminaries}

\subsection{Hilbert spaces induced by completely monotone kernels}\label{subsection: LiftHilbertSpace}

Let us consider, for a given function $g: \R_+ \longrightarrow \R$ and coefficients $b,\sigma: \R \longrightarrow \R$ that are at least measurable, the stochastic Volterra equation 
\begin{align}\label{eq: sve}
 X_t = g(t) + \int_0^t K(t-s)b(X_s)\, \mathrm{d}s + \int_0^t K(t-s)\sigma(X_s)\, \mathrm{d}B_s,\qquad t\in\R_+,
\end{align}
where it is implicitly assumed that both integrals are well-defined. Weak existence of solutions of such equations under suitable assumptions on the kernel $K$ follows from \cite{proemel_scheffels_weak} for regular $g$ and H\"older continuous $b,\sigma$, while for $g$ that may have a singularity at $t = 0$ we refer to \cite{BBF23} for the Lipschitz case, and \cite{CF24} for $b,\sigma$ being merely continuous with linear growth. In this section, we modify the Hilbert space-valued Markovian lift from~\cite{H23}; see also~\cite{F24}, where a similar construction was used. Here and below we suppose that the Volterra kernel, which may be singular in $0$ (in which case it is still well-defined on $\R_+$ in a $L_{\mathrm{loc}}^2$-sense, see Lemma \ref{lemma:liftkernelL2estimates}), satisfies the following assumption:

\begin{enumerate}
 \item[(A)] The Volterra kernel $K : \R_+ \longrightarrow \R_+$ is completely monotone and has representation $K(t) = K(\infty) + \int_{\R_+} e^{-xt} \mu(\mathrm{d}x)$, where $\mu$ is a Borel measure on $\R_+$ with $\mu(\{0\}) = 0$ and 
 \[
 \eta_* = \inf\left\{ \eta \in \R \ : \ \int_{\R_+} (1+x)^{-\eta} \, \mu(\mathrm{d}x) < \infty \right\} \in [-\infty, 1/2).
 \]
\end{enumerate}

Note that, by assumption, $\int_{\R_+}(1+x)^{\eta}\, \mu(\mathrm{d}x) < \infty$ holds for each $\eta < - \eta_*$. Given $\eta \in \R$, we denote by $\mathcal{H}_{\eta}$ the weighted Hilbert space of equivalence classes of functions $y: \R_+ \longrightarrow \R$ equipped with the inner product
\begin{align*}
    \langle y, \widetilde{y}\rangle_{\eta} = y(0)\hspace{0.02cm}\widetilde{y}(0) + \int_{\R_+} y(x)\hspace{0.02cm}\widetilde{y}(x)\hspace{0.02cm} (1 + x)^{\eta}\, \mu(\mathrm{d}x).
\end{align*}
Then it follows that $\mathcal{H}_{\eta} \subset \mathcal{H}_{\eta'}$ for $\eta' < \eta$, and it is easy to see that the operator $\Xi: \mathcal{H}_{\eta} \longrightarrow \R$ defined by
\begin{equation}\label{eq:projectionoperator}
    \Xi y = y(0) + \int_{\R_+} y(x)\, \mu(\mathrm{d}x)
\end{equation}
is a continuous linear functional on $\mathcal{H}_{\eta}$ whenever $\eta > \eta_*$. This functional has representation $\Xi y = \langle y, w_{\eta}\rangle_{\eta}$ with $w_{\eta} \in \mathcal{H}_{\eta}$ given by $w_{\eta}(x) = (1+x)^{-\eta}$, $x\in\R_+$. 

\begin{remark}\label{remark:mudefinitions}
    Besides the measure $\mu$, we will frequently work with the augmented measure $\overline{\mu}:=\delta_{0}+\mu$ where $\delta_{0}$ denotes the Dirac measure concentrated in $\{0\}$. The later allows us to express the inner product $\langle \cdot, \cdot \rangle_{\eta}$ and the action of the projection operator $\Xi$ in the convenient form 
    \begin{displaymath}
        \langle y, \widetilde{y}\rangle_{\eta} = \int_{\R_+} y(x)\widetilde{y}(x) (1 + x)^{\eta}\, \overline{\mu}(\mathrm{d}x) \quad\mbox{and}\quad \Xi y = \int_{\R_+} y(x)\, \overline{\mu}(\mathrm{d}x).
    \end{displaymath}
    In the original formulation \cite{H23}, the Markovian lift was constructed with respect to the Bernstein measure $\overline{\mu}_{K}:=K(\infty)\hspace{0.02cm}\delta_{0}+\mu$ of $K$ under the assumption $\eta_* \leq 1/2$. Our construction based on $\overline{\mu}$ allows us to capture time-invariant initial curves in the associated SVIE~\eqref{eq: sve}, i.e.\ $g\equiv x_0\in\R$, even if $K(\infty)=0$. Furthermore, our stronger assumption $\eta_* < 1/2$ allows us to construct a continuous Markovian lift in the domain of the projection operator $\Xi$, see Theorem \ref{thm: Markovian lift Markov property} below.
\end{remark}
Let $S(t)y(x) =e^{-tx}y(x)$. Then $(S(t))_{t \geq 0}$ defines a $C_0$-semigroup on $\mathcal{H}_{\eta}$ for $\eta \in \R$, and if $\eta' < \eta$ holds, then, using the inequality\begin{footnote}{The function $f(x) = (1+x)^{\eta - \eta'}e^{-2xt}$ attains its maximum at $x^* = (\eta - \eta')/(2t) - 1$. Consider first $t \geq (\eta - \eta')/2$, then $x^*\leq 0$, and monotonicity yields  $f(x) \leq f(0) = 1$ for $x \geq 0$. For $t \leq (\eta - \eta')/2$ we obtain $f(x) \leq f(x^*) \leq \kappa^2(\eta - \eta') t^{-(\eta - \eta')}$.}\end{footnote}
\[
    (1+x)^{\eta - \eta'}e^{-2xt} \leq \kappa^2(\eta - \eta')\big( 1 + t^{-(\eta-\eta')}\big), \quad
    x\in \R_+, t \in\R_+^*,
\]
with $\kappa(\delta) = \max\{1, 2^{-\delta/2} \delta^{\delta/2}\}$ in combination with $\sqrt{\cdot}$ being subadditive, we obtain $S(t) \in L(\mathcal{H}_{\eta'}, \mathcal{H}_{\eta})$, where for every $T\in\R_+^*$ we obtain the bound 
\begin{equation}\label{eq:liftsemigroupoperatornorm}
 \|S(t)\|_{L(\mathcal{H}_{\eta'}, \mathcal{H}_{\eta})} \leq C_{T}\hspace{0.03cm}\kappa(\eta - \eta') \hspace{0.03cm}t^{-(\eta-\eta')/2}\lesssim t^{-(\eta-\eta')/2}, \quad \ \forall t \in (0,T].
\end{equation}
Since in this work we exclusively apply the above estimates to bounded time intervals, we usually use the right-hand side of \eqref{eq:liftsemigroupoperatornorm} directly, thereby dropping the $T$-dependency of the constant for notational convenience.

The next proposition summarizes the properties of the composed operator $\Xi S(t)$, which will allow us to relate the Markovian lift with the original stochastic Volterra process.

\begin{proposition}\label{remark: markovian lift}
 Let $y \in \mathcal{H}_{\eta}$ with $\eta \in \R$. Then $g(t) = \Xi S(t)y$ is smooth on $\R_+^*$. If $\eta > \eta_*$, then $g$ is bounded on $\R_+$, while for $\eta \leq \eta_*$ we find for every $T\in\R_+^*$: 
 \[
    |g(t)| \leq C_T\,\kappa(\eta_* + \varepsilon - \eta )\|w_{\eta_* + \varepsilon}\|_{\eta_* + \varepsilon}\, t^{-(\eta_* + \varepsilon - \eta)/2}, \quad \ \forall t\in (0,T],
 \]
  and each $\varepsilon > 0$. Moreover, if $\int_{\R_+} (1+x)^{-\eta_*}\, \mu(\mathrm{d}x) < \infty$ holds, then we may even take $\varepsilon = 0$.
\end{proposition}
\begin{proof}
    Firstly, by dominated convergence, it is clear that $g$ is smooth on~$\R_+^*$. If $\eta > \eta_*$ holds, then combining the representation $\Xi  = \langle \cdot, w_{\eta}\rangle_{\eta}$ with the Cauchy-Schwarz inequality yields
    \[
        |g(t)| \leq \|w_{\eta}\|_{\eta} \|S(t)y\|_{\eta}
        \leq \left( 1+\int_{\R_+}(1+x)^{-\eta}\, \mu(\mathrm{d}x) \right)^{1/2} \|y\|_{\eta},
    \]
    which shows that $g$ is bounded on $\R_+$. For the case $\eta \leq \eta_*$, we may use \eqref{eq:liftsemigroupoperatornorm} with $\varepsilon > 0$ arbitrary, to find
    \begin{align*}
    |g(t)| &\leq \|w_{\eta_* + \varepsilon}\|_{\eta_* + \varepsilon} \|S(t)y\|_{\eta_* + \varepsilon}
    \\ &\leq C_T\,\kappa(\eta_* + \varepsilon - \eta)\left(1+ \int_{\R_+}(1+x)^{-\eta_* - \varepsilon}\, \mu(\mathrm{d}x) \right)^{1/2} \|y\|_{\eta} \,t^{-(\eta_*+\varepsilon - \eta)/2},
    \end{align*}
    which proves the second claim. Finally, if $\int_{\R_+} (1+x)^{-\eta_*}\, \mu(\mathrm{d}x) < \infty$, then letting $\varepsilon \searrow 0$, while utilizing $\lim_{\delta \to 0}\kappa(\delta) = 1$ for the case $\eta = \eta_*$, implies the last assertion.
\end{proof}

In terms of the operator $\Xi S(\cdot)$, the Volterra kernel $K$ has the representation
\begin{align}\label{eq: lift representation}
    K(t) = \Xi S(t) \xi_K \ \text{ with } \ \xi_K(x) = \1_{(0,\infty)}(x) + K(\infty)\1_{\{0\}}(x).
\end{align}
In particular, assumption (A) gives $\xi_K \in \mathcal{H}_{\eta}$ for each $\eta < - \eta_*$. Moreover, if $\eta_* < 0$, then we may choose $\eta \in (\eta_*, -\eta_*)$, and hence $K = \Xi S(\cdot)\xi_K$ is bounded by Proposition \ref{remark: markovian lift}. On the other hand, for $0 \leq \eta_* < \frac{1}{2}$, it follows that for each $\eta < -\eta_* < \eta_*$, the kernel $K$ satisfies the pointwise bound $K(t) \lesssim t^{- (\eta_* + \varepsilon - \eta)/2}$ for all $t \in (0, T]$ and $t\in\R_+^*$ with the same constant as given in Proposition \ref{remark: markovian lift}. Letting $\eta = -\eta_* - \varepsilon$ for some $\varepsilon > 0$ yields
\begin{align}\label{eq: K pointwise bound}
 K(t) \lesssim \|w_{\eta_* + \varepsilon}\|_{\eta_* + \varepsilon}\, t^{- \eta_* - \varepsilon}\lesssim t^{- \eta_* - \varepsilon}, \quad \ \forall t\in (0, T].
\end{align}
In particular, since $\eta_* < 1/2$, we may always find $\varepsilon > 0$ small enough such that $\eta_* + \varepsilon < 1/2$ and hence $K \in L_{\mathrm{loc}}^2(\R_+)$. The next lemma summarizes further useful properties of the Volterra kernel.

\begin{lemma}\label{lemma:liftkernelL2estimates}
    Suppose that condition~(A) holds for a Borel measure $\mu$ with $\eta_{*}<1/2$. Then the associated Volterra kernel, assuming $K\not\equiv 0$, satisfies for every $T\in\R_+^*$, $h \in (0,T]$ and $\varepsilon\in (0,1-2\eta_{*})$:
    \begin{align*}
        C\big(\mu((0,1/h])\vee 1\big)^2 h\leq\int_0^h K(r)^2\,\mathrm{d}r
        \leq \overline{C} \cdot \begin{cases} h, & \eta_* < 0
        \\ h^{1 - 2\eta_* - \varepsilon}, & 0 \leq \eta_* < \frac{1}{2}. \end{cases}
    \end{align*}
    Moreover, it holds that
    \begin{align*}
        \int_0^T |K(h + r) - K(r)|^2\,\mathrm{d}r
        \leq \widetilde{C} \cdot \begin{cases}h, & \eta_* < 0 
        \\ h^{1-2\eta_{*}-\varepsilon}, & 0 \leq \eta_* < \frac{1}{2}, \end{cases}
    \end{align*}
    where the constants $C, \overline{C}, \widetilde{C}>0$ may depend on $\varepsilon$ for $\eta_{*}>-1$. Again, taking $\varepsilon = 0$ is admissible, if $\int_{\R_+} (1+x)^{-\eta_*}\, \mu(\mathrm{d}x) < \infty$.
\end{lemma}
Since the proof of this lemma is quite technical and does not provide additional insights for the study of small-time central limit theorems, it is given in Appendix \ref{section:appendixkernelestimates} of this work. As a consequence of Lemma \ref{lemma:liftkernelL2estimates}, let us note that a completely monotone kernel $K \not \equiv 0$ that satisfies condition (A), automatically fulfils \eqref{eq:kernelL2order} and \eqref{eq:kernelincrementL2} on every time interval $[0, T]$ with $T\in\R_+^*$.

In the case $\eta_{*}<0$, $\mu$ is a finite measure and hence the order of the lower bound in Lemma \ref{lemma:liftkernelL2estimates} is sharp. However, sharp upper and lower bounds are more delicate when $\eta_{*}\in [0,1/2)$. Ideally, a lower bound in Lemma \ref{lemma:liftkernelL2estimates} should take the form $\mu((0,1/h])\gtrsim h^{-\eta_{*}}$. This combined with the other two bounds holding for $\varepsilon=0$ would be convenient for obtaining an fCLT (cf.\ 
Corollary~\ref{corollary:functionalCLT}), and holds for many examples of interest (see
Subsection~\ref{subsection:CLTKernelExamples}).
Unfortunately, it turns out that, in general, this does not need to be the case, as illustrated in the following example.

\begin{example}\label{example:kernelestimeslowercannotimproved}
    Let $\eta_* \in (0,1/2)$ be arbitrary and $\beta \in (0,\eta_*)$. Then there exists a Borel measure $\mu$ on $\R_+$ that satisfies condition (A) with $\eta_*$ for which one has
    \[
        \liminf_{x \to \infty} x^{-\beta'}\mu((0,x]) = 0, \qquad \forall \beta' \in (\beta, \eta_*).
    \]  
    Indeed, fix $\beta \in(0,\eta_*)$ and define $x_k := \exp\big((2\eta_*/\beta)^k\big)$ for $k \geq 0$. Let $\mu_{F}$ be a Borel measure on~$\mathbb{R}_+$ defined via its distribution function~$F(x)=\mu_{F}([0,x])$ given by
    \begin{equation}\label{eq:kernellowerboundcounterexampledistrfct}
    F(x):=
    \begin{cases}
    0, & 0\leq x<1, \\
    x_k^{\eta_*}, &  x_k \leq x < x_k^{\eta_*/\beta},\ k \in \mathbb{N}, \\
      x^{\eta_*}, & x_k^{\eta_*/\beta} \leq x < x_{k+1},\ k \in \mathbb{N}_0.
    \end{cases}
    \end{equation}
    Then $F(x)\leq x^{\eta_*}$ for all $x\in\R_+$ and, using $F|_{[0,1)}\equiv 0$ combined with the monotonicity of $F$, it follows that $\int_0^\infty x^{-\eta_*-\varepsilon} \,\mathrm{d}F(x) < \infty$ holds for every $\varepsilon>0$. On the other hand, we can estimate
    \begin{align*}
        \int_0^\infty x^{-\eta_*} \,\mathrm{d}F(x) &\gtrsim \sum_{k=1}^\infty \int_{x_k^{\eta_*/\beta}}^{ x_{k+1}} \frac{\mathrm{d}x}{x} 
      = \frac{\eta_*}{\beta} \sum_{k=1}^\infty \Big( \frac{2\eta_*}{\beta} \Big)^k  = \infty.
    \end{align*}
    Therefore, the Borel measure $\mu_{F}$ as well as the associated completely monotone kernel induced by~$F$ satisfy condition~(A) with~$\eta_{*}$, since $\supp(\mu_{F})\subseteq [1, \infty)$ and $(1+x^{-1})^{-\eta}\in [2^{-\eta}, 1]$ for every $x\ge 1$ and $\eta>0$. Finally, it is an immediate consequence of \eqref{eq:kernellowerboundcounterexampledistrfct} that bounding $\mu_{F}([0,x])=F(x)$ below by a power function with order $\beta'>\beta$ is impossible since for $x=x_k^{\eta_*/\beta}$, $k\in\mathbb N$, it follows $F(x-)=x^\beta$, while $F$ is constant in a left neighborhood of~$x$.
\end{example}
For positive results under further assumptions on~$\mu$, we refer to~\cite{BiGoTe87} and~\cite[p.~112]{tenenbaum1995introduction}.
For instance, it follows from Karamata's Tauberian theorem (see \cite[Theorem 1.7.1]{BiGoTe87}), that 
\[
    K(t) \sim C\hspace{0.02cm} t^{\alpha - 1}\hspace{0.02cm}\ell(t^{-1}), \qquad t \to 0,
\]
for some $\alpha \in (0,1]$ and a slowly varying function~$\ell$, is equivalent to
\[
    \mu([0,x]) \sim \frac{C}{\Gamma(2 - \alpha)}\hspace{0.02cm}x^{1 - \alpha} \hspace{0.02cm}\ell(x), \qquad x \to \infty.
\]
As a special case, this contains the rough versions, i.e.\ $H\in (0, 1/2)$, of the Riemann-Liouville kernel, where $\eta_*=1/2-H$ and all three bounds in Lemma \ref{lemma:liftkernelL2estimates} have the same order (see also Example \ref{example:bernsteinmeasures} (a) below), and the $\log$-modulated kernel (see Example \ref{example:kernelssatisfying CLT}).

\subsection{The Markovian lift}\label{subsection:MarkovianLiftGeneral}

Let us introduce the class of admissible functions $g$ appearing in \eqref{eq: sve} as the image under the operator $\Xi S(\cdot)$, i.e., for $\eta \in \R$ we let
\[
 \mathcal{G}_{\eta} = \left\{ g: \R_+ \longrightarrow \R \ : \ \exists \xi \in \mathcal{H}_{\eta} \ \text{ s.t. } g(\cdot) = \Xi S(\cdot)\xi  \right\}.
\]
By Proposition~\ref{remark: markovian lift}, each $g \in \mathcal{G}_{\eta}$ is smooth with possibly a singularity in $t = 0$. Moreover, Proposition~\ref{remark: markovian lift} implies that $\mathcal{G}_{\eta} \subset L_{\mathrm{loc}}^{p}(\R_+)$ with $p = \infty$ when $\eta > \eta_*$, and $1 \leq p < \frac{2}{\eta_* - \eta}$, if $\eta_* - 2 < \eta < \eta_*$. For given $g = \Xi S(\cdot)\xi_g \in \mathcal{G}_{\eta}$ with $\eta$ fixed, let us consider the stochastic Volterra equation \eqref{eq: sve}. The corresponding Markovian lift $(\mathcal{X}_t)_{t \geq 0}$ is obtained by imposing the requirement that $\Xi \mathcal{X}_t = X_t$. Since $K(t) = \Xi S(t)\xi_K$, and formally interchanging the operator $\Xi$ with both integrals, we necessarily arrive at the representation
\begin{align}\label{eq: markovian lift}
 \mathcal{X}_t = S(t)\xi_g + \int_0^t S(t-s)\xi_K b(\Xi\mathcal{X}_s)\, \mathrm{d}s + \int_0^t S(t-s)\xi_K \sigma(\Xi\mathcal{X}_s)\, \mathrm{d}B_s, \quad t\in\R_+.
\end{align}
Note that this equation is a mild formulation of a stochastic evolution equation (SEE) on the Hilbert space $\mathcal{H}_{\eta}$ with the generator determined by the semigroup $(S(t))_{t \geq 0}$ (cf.\ \cite[Eq.~(2.7)]{H23}). 
We now state an existence result for \eqref{eq: sve} and its Markovian lift that is sufficient for our purposes.
\begin{theorem}\label{thm: Markovian lift Markov property}
 Suppose that condition (A) is satisfied with $\eta_* < 1/2$, and that $b, \sigma$ are continuous with linear growth. Then for each $g = \Xi S(\cdot)\xi_g \in \mathcal{G}_{\eta_g}$ with $\eta_g > \eta_*$, there exists a continuous weak solution of \eqref{eq: sve}. Moreover, there also exists a weak solution $\mathcal{X} \in L^2(\Omega, \P; C([0,T]; \mathcal{H}_{\eta}))$ of \eqref{eq: markovian lift} with $\eta \in (\eta_*, 1 - \eta_*)$ such that $\Xi \mathcal{X} = X$ holds on $[0,T]$ for arbitrary $T > 0$.
\end{theorem}
\begin{proof} 
 Firstly, it follows from Proposition \ref{remark: markovian lift} that $g$ is bounded on $\R_+$ and smooth on~$\R_+^*$. Then, Lemma \ref{lemma:liftkernelL2estimates} implies that \cite[Theorem 2.6]{CF24} is applicable, which yields the weak existence of a continuous weak solution $X$ of \eqref{eq: sve}. Since both $b$ and $\sigma$ are continuous and of linear growth, an application of Lemma \ref{lemma: regularization lift} gives the existence of a weak solution of \eqref{eq: markovian lift} with $\mathcal{X} \in L^2(\Omega, \P; C([0,T]; \mathcal{H}_{\eta}))$ and $X = \Xi \mathcal{X}$ on $[0,T]$. Since $T > 0$ was arbitrary, the assertion is proved.
\end{proof}

Note that, if uniqueness in law holds for \eqref{eq: markovian lift}, then under the conditions of the above theorem, \eqref{eq: markovian lift} determines a Markov process, which justifies the notion of Markovian lift. For further details in this direction we refer to the second part of \cite[Lemma 4.3]{hamaguchi2023weak} for a general proof concept under weak uniqueness and \cite[Section 2]{H23} where, in particular, the $C_b$-Feller property was shown for Lipschitz continuous coefficients. 

Let us close this section with a few examples of popular Volterra kernels $K$ that admit an explicit formula for $\mu$ and the Bernstein measure $\mu_K$.

\begin{example}\label{example:bernsteinmeasures}
    \begin{enumerate}
        \item[(a)] Suppose that $K(t) = \frac{t^{\alpha - 1}e^{-\beta t}}{\Gamma(\alpha)}$ with $\alpha \in (1/2,1)$ and $\beta \geq 0$, which covers both Riemann-Liouville and gamma kernels with $\alpha=H+1/2$. Then it follows that $K(\infty)=0$ and $\eta_* = 1 - \alpha \in (0,1/2)$ with 
        \[
            \mu(\mathrm{d}x) = \frac{(x-\beta)^{-\alpha}}{\Gamma(1-\alpha)\Gamma(\alpha)}\mathbbm{1}_{(\beta, \infty)}(x)\,\mathrm{d}x.
        \]  

        \item[(b)] Let $K(t) = \log(1 + 1/t)$, then $K(\infty)=0$ and $\eta_* = 0$ with 
        \[
            \mu(\mathrm{d}x) = \frac{1 - e^{-x}}{x}\,\mathrm{d}x.
        \]

        \item[(c)] Let $K(t) = c_0 + \sum_{i=1}^N c_i e^{-\lambda_i t}$ with $c_0, c_1,\dots, c_N \geq 0$, $N \geq 1$, and $\lambda_1,\dots, \lambda_N > 0$. Then $\eta_* = - \infty$, $K(\infty)=c_0$ and $\mu$ is given by
        \[
            \mu(\mathrm{d}x) = \sum_{i=1}^N c_i\hspace{0.02cm} \delta_{\lambda_i}(\mathrm{d}x),
        \]
        where $\delta_{w}$ denotes the Dirac measure concentrated in $\{w\}$.

        \item[(d)] For every $K$ satisfying condition (A) and $\varepsilon>0$, the shifted kernel defined via $K_{\varepsilon}:=K(\cdot +\varepsilon)$ fulfills again condition (A) with $\eta_*=-\infty$, $K_{\varepsilon}(\infty)=K(\infty)$, and $\mu_{\varepsilon}\ll\mu$ is given by 
        \begin{displaymath}
            \mu_{\varepsilon}(\mathrm{d}x)=e^{-\varepsilon x}\,\mu(\mathrm{d}x).
        \end{displaymath}
    \end{enumerate}
\end{example}

\section{Small-time CLTs for the Markovian lift}\label{section:smalltimeCLTLift}

\subsection{A CLT for the finite-dimensional distributions}

In this section we prove a small-time central limit theorem for the finite-dimensional projections of the Markovian lift based on continuous linear functionals, which may be written as $\langle\cdot,y\rangle_{\eta}$ for some $y\in\mathcal{H}_{\eta}$ by the Riesz representation theorem. Thus, as a preliminary step, let us first recall an auxiliary result stating that the Brownian integral processes on $\mathcal{H}_{\eta}$ that we are going to encounter in our arguments (cf.~\eqref{eq:LiftCLTmultivarGaussian} and \eqref{eq: 3}) are indeed Gaussian.

\begin{lemma}\label{lemma: Gaussian integral}
    Suppose that condition (A) is satisfied and let $K$ have representation \eqref{eq: lift representation}. Then  
    \[
        I(t) = \int_0^t S(t-s)\xi_K\, \mathrm{d}B_s, \quad t \in\R_+,
    \]
    defines a continuous Gaussian process on $\mathcal{H}_{\eta}$ with $\eta < 1-\eta_*$. In particular, considering a collection of positive time points $(t_j)_{j\in\{1,\dots,N\}}$ and $y_1,\dots, y_N \in \mathcal{H}_{\eta}$, then the $N$-dimensional random vector $(\langle I(t_j), y_j\rangle_{\eta})_{j=1,\dots,N}$ is Gaussian with mean zero and covariance structure
    \begin{align*}
         \E\left[ \langle I(t_j), y_j\rangle_{\eta}\hspace{0.02cm} \langle I(t_k), y_k\rangle_{\eta} \right]
         = \int_{0}^{t_j \wedge t_k} \langle S(t_j - r)\xi_K, y_j \rangle_{\eta}\hspace{0.02cm} \langle S(t_k - r)\xi_K, y_k\rangle_{\eta}\, \mathrm{d}r.
    \end{align*}
\end{lemma}
\begin{proof} 
 Firstly, since $\xi_K \in \mathcal{H}_{\eta'}$ for $\eta' < - \eta_*$, it follows by standard integration theory that $I$ is a Gaussian process on $\mathcal{H}_{\eta'}$, see \cite[Theorem 5.2]{DaPrato_Zabczyk_2014}. Lemma \ref{lemma: regularization lift} then implies $I \in L^2(\Omega, \P; C([0,T]; \mathcal{H}_{\eta}))$ for each $T > 0$ as the solution to the associated SVIE is uniquely given by $K\ast \mathrm{d}B$. In particular, the $N$-dimensional random vector $(\langle I(t_j), y_j\rangle_{\eta})_{j\in\{1,\dots,N\}}$ is Gaussian with mean zero and the stated covariance structure. 
\end{proof}

The following is our first main result on the small-time central limit theorem for finite-dimensional distributions of the Markovian lift. 

\begin{theorem}\label{theorem:LiftClt}
 Suppose that condition (A) is satisfied, and let $b \in C^{\chi_b}(\R)$ and $\sigma \in C^{\chi_{\sigma}}(\R)$ for some $\chi_b, \chi_{\sigma} \in (0,1]$. Let $K$ have representation \eqref{eq: lift representation} and let~$\mathcal{X}$ be any continuous weak solution of \eqref{eq: markovian lift} on $\mathcal{H}_{\eta_{\mathcal{X}}}$ with $\eta_{\mathcal{X}} \in (\eta_*, 1 - \eta_*)$ and $g = \Xi S(t)\xi_g$, where $\xi_g \in \mathcal{H}_{\eta_g}$ for some $\eta_g > \eta_*$, and define $\chi_g := (\eta_g - \eta_*)/2$. Fix $N \geq 1$, and let $y_1,\dots, y_N \in \mathcal{H}_{\eta}\setminus\{0\}$ with $\eta\le \eta_{\mathcal{X}}$ and $0 < t_1 \le \dots \le t_N$ satisfy the following conditions:
   \begin{enumerate}
        \item[(i)] There exist $C, \gamma_* > 0$ such that 
        \[
            \lambda_{i}(n) := \left( \int_0^{1/n} \langle S(r)\xi_{K}, y_i\rangle_{\eta}^2\, \mathrm{d}r \right)^{-1}
        \]
        satisfies $\sqrt{\lambda_{i}(n)} \leq C n^{\gamma_*}$ for all $i \in\{1,\dots, N\}$, and
        \begin{equation}\label{eq:LiftCLTexponentcorridorgeneral}
            \gamma_{*} < \frac{\min\{1, 1-\eta - \eta_*\}}{2} + \min\left\{\left( \frac{1}{2} - \eta_*^+\right), \chi_g\right\}\chi_{\sigma}.
        \end{equation}
        
        \item[(ii)] There exists a symmetric, positive semi-definite $N\times N$-matrix $\Sigma$ such that for all $i,j \in\{1,\dots, N\}$ with $i \le j$:
        \begin{align*}
        \notag \lim_{n \to \infty} &\sqrt{\lambda_{i}(n)\hspace{0.02cm}\lambda_{j}(n)} \int_0^{t_i/n} \langle S((t_j-t_i)/n+r)\xi_{K}, y_j\rangle_{\eta}\, \langle S(r)\xi_{K}, y_i\rangle_{\eta}\,\mathrm{d}r
        = \Sigma_{ij}.
        \end{align*}
    \end{enumerate}
  Then, as $n\rightarrow\infty$, we obtain
 \begin{equation}\label{eq:LiftCLT}
  \left( \sqrt{\lambda_i(n)} \left( \langle \mathcal{X}_{t_i/n} - S(t_{i}/n)\xi_g, y_i\rangle_{\eta} \right) \right)_{i=1,\dots,N} \stackrel{d}{\longrightarrow} \mathcal{N} \big(0, \sigma(\Xi \xi_g)^2 \Sigma \big).
 \end{equation}
\end{theorem}
\begin{proof}
    First, recall that we have $\mathcal{H}_{\eta_{\mathcal{X}}}\subseteq\mathcal{H}_{\eta}$ by $\eta\le \eta_{\mathcal{X}}$, whence $\mathcal{X}$ is, in particular, $\mathcal{H}_{\eta}$-valued. Analogously to \eqref{eq:generalSVIEdividedinYZinitvarying} in Theorem \ref{theorem:varyinginitialCLT}, let us write $\mathcal{X}_{t} = S(t)\xi_g + \mathcal{Y}_t + \mathcal{Z}_t$, where $\mathcal{Y}_t = \mathcal{Y}_t^b + \mathcal{Y}_t^{\sigma}$ denotes the Gaussian part, and $\mathcal{Z}_t = \mathcal{Z}_t^b + \mathcal{Z}_t^{\sigma}$ the remainder, given by 
    \begin{align*}
      \mathcal{Y}^b_{t} &= \int_0^t S(t-s)\hspace{0.02cm}\xi_{K}\hspace{0.02cm}b(\Xi\xi_{g})\, \mathrm{d}s, 
      \\ \mathcal{Z}_t^b &= \int_0^t S(t-s)\hspace{0.02cm}\xi_{K}\hspace{0.02cm}(b(\Xi\mathcal{X}_{s})-b(\Xi\xi_{g}))\, \mathrm{d}s,
      \\ \mathcal{Y}_t^{\sigma} &= \int_0^t S(t-s)\hspace{0.02cm}\xi_{K}\hspace{0.02cm}\sigma(\Xi\xi_{g})\, \mathrm{d}B_s, 
      \\ \mathcal{Z}_t^{\sigma} &= \int_0^t S(t-s)\hspace{0.02cm}\xi_{K}\hspace{0.02cm}(\sigma(\Xi\mathcal{X}_{s})-\sigma(\Xi\xi_{g}))\, \mathrm{d}B_s.
    \end{align*}

    \textit{Step 1.} Proceeding similarly to the finite-dimensional case from Section~\ref{section:smalltimeCLTSVIE}, let us first bound the moments of~$\mathcal{Z}_t$. For this purpose, note that by the commutativity of the continuous linear functional $\Xi$, following from $\eta_{\mathcal{X}}>\eta_*$, with $\mathcal{H}_{\eta_{\mathcal{X}}}$-valued Bochner and stochastic integrals and the special form of $g$, it follows that $X:=\Xi\mathcal{X}$ defines a continuous solution to \eqref{eq: sve}. Moreover, by \eqref{eq: K pointwise bound} combined with $\eta_* < \frac{1}{2}$, we find $\varepsilon > 0$ small enough such that 
    \begin{align}\label{eq: 1}
        \int_0^t K(s)^2\, \mathrm{d}s \lesssim t^{2 \gamma} \ \text{ with } \ 0<\gamma = \begin{cases}\frac{1}{2} - \eta_* - \varepsilon, & \eta_* \in [0, 1/2)
        \\ \frac{1}{2}, & \eta_* < 0. \end{cases}
    \end{align}
     Hence, combining this with the boundedness of $g$ due to $\eta_g>\eta_*$ and Proposition \ref{remark: markovian lift} shows that Lemma \ref{lemma:momentestimatewithg} is applicable. Furthermore, we obtain for $\overline{\chi}_g:=\min\{\chi_g,1/2\}>0$ with~$\chi_g$ introduced above and $\varepsilon_g \in (0,\overline{\chi}_g)$ for each $s \geq 0$ the bound
\begin{align}\label{eq:gquasiHölder}
    \notag |g(s) - g(0)| &\leq \int_{\R_+} (1 - e^{-xs})|\xi_g(x)|\, \mu(\mathrm{d}x)
    \\ &= s^{ \overline{\chi}_g - \varepsilon_g} \int_{\R_+} x^{\overline{\chi}_g - \varepsilon_g} |\xi_g(x)|\, \mu(\mathrm{d}x)
    \\\notag &\leq s^{ \overline{\chi}_g - \varepsilon_g} \| \xi_g\|_{\eta_g} \left( \int_{\R_+} (1+x)^{-\eta_* - 2\varepsilon_g}\, \mu(\mathrm{d}x) \right)^{1/2}\lesssim s^{ \overline{\chi}_g - \varepsilon_g},
\end{align}
    where we have used $\overline{\chi}_g\le 1/2$, the H\"older continuity of $e^{-(\cdot)}$ on $\R_+$, $\overline{\chi}_g\le(\eta_g - \eta_*)/2$ and finally Hölder's inequality. Clearly, the factors independent of $s$ are bounded due to $\xi_g \in \mathcal{H}_{\eta_g}$ and the definition of $\eta_*$. Moreover, since $\eta < 1 - \eta_*$, we find $\delta \in (0,1)$ small enough such that $\delta < 1 \wedge ( 1 - \eta_* - \eta)$. Define $\eta' = \eta - (1-\delta)$. Then $\eta' < - \eta_*$ and hence $\xi_{K}\in\mathcal{H}_{\eta'}$. Therefore, \eqref{eq:liftsemigroupoperatornorm} applied to $\eta' < \eta$ implies for every $t>0$:
    \begin{align}\label{eq: good bound on SK delta generalized}
        \|S(r)\xi_K\|_{\eta} \lesssim r^{-\frac{1-\delta}{2}}, \qquad \forall r \in (0,t],
    \end{align}
    which lies in $L_{\mathrm{loc}}^{2}(\R_{+})$, since $\delta \in (0,1)$. Hence, using the Jensen inequality combined with \eqref{eq: good bound on SK delta generalized}, the Hölder continuity of $b$ and $\sigma$, $\Xi \xi_g=g(0)$, \eqref{eq:momentestimatewithg} from Lemma \ref{lemma:momentestimatewithg} and the Hölder-type estimate~\eqref{eq:gquasiHölder}, we arrive at
    \begin{align*}
        \E\left[ \|\mathcal{Z}_t\|_{\eta}^p \right]
        &\lesssim \E\left[ \|\mathcal{Z}^b_t\|_{\eta}^p \right] + \E\left[ \|\mathcal{Z}_t^{\sigma}\|_{\eta}^p \right]
        \\ &\lesssim \left( \int_0^t \|S(r)\xi_K\|_{\eta}\, \mathrm{d}r \right)^{p-1}\int_0^t \|S(t-s)\xi_K\|_{\eta}\hspace{0.04cm} \E\big[ |X_s - g(0)|^{p \chi_b} \big]\, \mathrm{d}s
        \\ &\qquad + \left( \int_0^t \|S(r)\xi_K\|_{\eta}^2\, \mathrm{d}r\right)^{\frac{p}{2}-1}\int_0^t \|S(t-s)\xi_K\|_{\eta}^2 \hspace{0.04cm}\E\big[ |X_s- g(0)|^{p \chi_{\sigma}} \big]\, \mathrm{d}s
        \\ &\lesssim t^{(p-1)\frac{1 + \delta}{2}}\int_0^t (t-s)^{-\frac{1-\delta}{2}} \left(s^{p\chi_b\gamma}+s^{p\chi_b(\overline{\chi}_g-\varepsilon_g)}\right)\, \mathrm{d}s
        \\ &\qquad + t^{(p/2 - 1)\delta}\int_0^t (t-s)^{-(1 - \delta)} \left(s^{p \chi_{\sigma}\gamma}+s^{p \chi_{\sigma}(\overline{\chi}_g-\varepsilon_g)}\right) \, \mathrm{d}s
        \\ &\lesssim t^{p \left(\frac{1+\delta}{2} + \chi_b \min\{\gamma, \overline{\chi}_g-\varepsilon_g\}\right)} + t^{p\left( \frac{\delta}{2} + \chi_{\sigma}\min\{\gamma, \overline{\chi}_g-\varepsilon_g\} \right)},
    \end{align*}
    where $p > \max\{ 2/\chi_b, 2/\chi_{\sigma}\}$. Using this estimate, let us show that it suffices to study the convergence of the Gaussian part $\mathcal{Y}$. First, define $\mathcal{Z}_{n}^{*}:=\max_{i\in\{1,\dots,N\}}\|\mathcal{Z}_{t_{i}/n}\|_{\eta}$. Then, using the definitions of~$\mathcal{Y}$ and~$\mathcal{Z}$, we arrive at
    \begin{align}
        \notag &\mathbb{P}\left[\bigcap_{i=1}^{N}\left\{\sqrt{\lambda_{i}(n)}\big\langle\mathcal{X}_{t_{i}/n}-S(t_{i}/n)\xi_{g},y_{i}\big\rangle_{\eta}\le a_{i}\right\}\right]
        \\ &= \mathbb{P}\left[\bigcap_{i=1}^{N}\left\{\sqrt{\lambda_{i}(n)}\big\langle\mathcal{Y}_{t_{i}/n}+\mathcal{Z}_{t_{i}/n},y_{i}\big\rangle_{\eta}\le a_{i}\right\}\cap\big\{\mathcal{Z}_{n}^{*} \le z(t_{N}/n)\big\}\right] \label{eq: first probability}
        \\ &\hspace{0.45cm} + \mathbb{P}\left[\bigcap_{i=1}^{N}\left\{\sqrt{\lambda_{i}(n)}\big\langle\mathcal{Y}_{t_{i}/n}+\mathcal{Z}_{t_{i}/n},y_{i}\big\rangle_{\eta}\le a_{i}\right\}\cap\big\{\mathcal{Z}_{n}^{*} > z(t_{N}/n)\big\}\right] \notag
    \end{align}
    where $(a_{1},\dots,a_{N})^{\intercal}\in\mathbb{R}^{N}$ is any point of continuity of the distribution function of $\mathcal{N}(0, \sigma(\Xi\xi_g)^2 \Sigma)$, and $z(t) = t^{q}$, where the exponent satisfies 
    \begin{equation}\label{eq:LiftCLTchoiceofq}
        q \in \left( \gamma_{*},\ \frac{\delta}{2} + \min\left\{\left( \frac{1}{2} - \eta_*^+ - \varepsilon\right), \overline{\chi}_g-\varepsilon_g\right\}\chi_{\sigma} \right).
    \end{equation}
    Note that by assumption (i) and the particular form of~$\gamma\le 1/2$ given
    by~\eqref{eq: 1} combined with $\overline{\chi}_g=\min\{\chi_g,1/2\}$, this interval is non-empty, provided that $\varepsilon$ and $\varepsilon_g$ are small enough and $\delta$ is chosen to be close enough to $1 \wedge (1 - \eta_* - \eta)$. Moreover, since $\gamma_* > 0$, the function $z$ is nondecreasing. In particular, the second probability above tends to zero as $n\rightarrow\infty$ since we can estimate for every $p > \max\{ 2/\chi_b, 2/\chi_{\sigma}\}$:  
    \begin{align}\label{eq:MarkovineqforMaxCLTlift}
        \notag \mathbb{P}\big[\mathcal{Z}_{n}^{*} > z(t_{N}/n)\big] 
        &\leq \frac{\sum_{i=1}^{N}\mathbb{E}\big[\|\mathcal{Z}_{t_{i}/n}\|_{\eta}^{p}\big]}{z(t_{N}/n)^{p}}
        \\ \notag &\lesssim N\frac{(t_{N}/n)^{\min\left\{ \frac{1+\delta}{2}+ \min\{\gamma, \overline{\chi}_g-\varepsilon_g\}\chi_b,  \ \frac{\delta}{2} + \min\{\gamma, \overline{\chi}_g-\varepsilon_g\}\chi_{\sigma} \right\}p}}{z(t_{N}/n)^{p}}\\
        &\lesssim n^{-\left(\frac{\delta}{2} + \min\{\gamma, \overline{\chi}_g-\varepsilon_g\}\chi_{\sigma}-q\right) p}\longrightarrow 0,
    \end{align}
    where we have utilized $q < \frac{\delta}{2} + \min\{\gamma, \overline{\chi}_g-\varepsilon_g\}\chi_{\sigma}  \leq \frac{1+\delta}{2}$ following from \eqref{eq:LiftCLTchoiceofq}, $\chi_\sigma\le 1$, $\gamma\le 1/2$ and $\overline{\chi}_g\le 1/2$. Thus, it remains to study the convergence of the first probability in \eqref{eq: first probability}.
    
    \textit{Step 2.} In this step we prove an asymptotic upper bound for \eqref{eq: first probability}. Using on $\{\mathcal{Z}_{n}^{*}\le z(t_{N}/n)\}$ the inequality $|\langle\mathcal{Z}_{t_{i}/n},y_{i}\rangle_{\eta}| \leq z(t_{N}/n)\hspace{0.05cm}\|y_{i}\|_{\eta}$, we obtain 
    \begin{align}\label{eq:upperestimateCLTlift}
            \notag &\mathbb{P}\left[\bigcap_{i=1}^{N}\left\{\sqrt{\lambda_{i}(n)}\big\langle\mathcal{Y}_{t_{i}/n}+\mathcal{Z}_{t_{i}/n},y_{i}\big\rangle_{\eta}\le a_{i}\right\}\cap\big\{\mathcal{Z}_{n}^{*} \le z(t_{N}/n)\big\}\right]
            \\ \notag &\le\mathbb{P}\left[\bigcap_{i=1}^{N}\left\{\sqrt{\lambda_{i}(n)}\big(\big\langle\mathcal{Y}_{t_{i}/n},y_{i}\big\rangle_{\eta}-z(t_{N}/n)\|y_{i}\|_{\eta}\big)\le a_{i}\right\}\right]
            \\ &= \mathbb{P}\left[\bigcap_{i=1}^{N}\left\{\sqrt{\lambda_{i}(n)}\big(\big\langle\mathcal{Y}_{t_{i}/n}^{\sigma},y_{i}\big\rangle_{\eta}+\mu_{i}(n)\big)\le a_{i}\right\}\right],
    \end{align}
    where we have set $\mu_{i}(n):= \big\langle\mathcal{Y}_{t_{i}/n}^{b},y_{i}\rangle_{\eta}-z(t_{N}/n)\hspace{0.05cm}\|y_{i}\|_{\eta}$ for every $i \in\{1,\dots, N\}$. Note that, by Lemma~\ref{lemma: Gaussian integral}, 
    \begin{equation}\label{eq:LiftCLTmultivarGaussian}
        \left(\sqrt{\lambda_{i}(n)}\big(\big\langle\mathcal{Y}_{t_{i}/n}^{\sigma},y_{i}\big\rangle_{\eta}+\mu_{i}(n)\big), \hspace{0.2cm}i\in\{1,\dots,N\}\right)^{\intercal}
    \end{equation}
    is an $N$-dimensional Gaussian random vector for every $n\in\mathbb{N}$. In particular, the probability given in \eqref{eq:upperestimateCLTlift} corresponds to its distribution function evaluated at $(a_{1},\dots,a_{N})^{\intercal}$.    
    Therefore, by Lévy's continuity theorem and the continuity of the characteristic function of the multivariate Gaussian distribution in its parameters, holding even in the degenerate case, it is sufficient to study the convergence of the mean and covariance matrix. For the mean, we estimate for every $n\in\mathbb{N}$:
    \begin{align*}
        \notag \sqrt{\lambda_{i}(n)}\,|\mu_{i}(n)|
        &\le\sqrt{\lambda_{i}(n)}\,\left(\|\mathcal{Y}_{t_{i}/n}^{b}\|_{\eta}+z(t_{N}/n)\right)\,\|y_{i}\|_{\eta}
        \\ &\lesssim \sqrt{\lambda_{i}(n)}\,\left(\int_{0}^{\frac{t_{i}}{n}}\big\|S\big(\tfrac{t_{i}}{n}-s\big)\xi_{K}\big\|_{\eta}\, \mathrm{d}s + z(t_{N}/n)\right)
        \\ &\lesssim n^{\gamma_* - \frac{1+\delta}{2}}+n^{\gamma_*-q},
   \end{align*}
   where we have used the Cauchy-Schwarz inequality, the estimate from \eqref{eq: good bound on SK delta generalized}, $\sqrt{\lambda_{i}(n)} \lesssim n^{\gamma_{*}}$ by condition (i) and $z(t_{N}/n) = (t_N/n)^{q}$. Due to condition (i) and the choice of~$\delta$ and~$q$, the right-hand side converges to zero as $n \to \infty$. For the covariance matrix we observe for every $i,j\in\{1,\dots,N\}$ with $i\le j$ by an application of condition (ii): 
    \begin{align*}
    &\lim_{n\rightarrow\infty}\,\mathrm{cov}\left( \sqrt{\lambda_{i}(n)}\hspace{0.03cm}\big\langle\mathcal{Y}_{t_{i}/n}^{\sigma},y_i\big\rangle_{\eta}, \sqrt{\lambda_{j}(n)}\hspace{0.03cm}\big\langle\mathcal{Y}_{t_{j}/n}^{\sigma},y_j\big\rangle_{\eta}\right)
    \\ &= \sigma(\Xi \xi_g )^2\lim_{n\rightarrow\infty} \sqrt{\lambda_{i}(n)\hspace{0.03cm}\lambda_{j}(n)} \int_0^{t_{i} /n} \langle S((t_{j}-t_i)/n+r)\xi_{K}, y_j\rangle_{\eta}\, \langle S(r)\xi_{K}, y_i\rangle_{\eta}\, \mathrm{d}r
    \\ &= \sigma(\Xi \xi_g)^2 \Sigma_{ij}.
    \end{align*}
    This proves the weak convergence of \eqref{eq:LiftCLTmultivarGaussian} towards $\mathcal{N}\big(0, \sigma(\Xi \xi_g)^2 \Sigma \big)$. Since $(a_{1},\dots,a_{N})^{\intercal}$ is a continuity point of the distribution function of the latter, using the upper bound
    in~\eqref{eq:upperestimateCLTlift} we obtain
    \begin{align*}
    \limsup_{n \to \infty}\hspace{0.1cm} &\mathbb{P}\left[\bigcap_{i=1}^{N}\left\{\sqrt{\lambda_{i}(n)}\big\langle\mathcal{Y}_{t_{i}/n}+\mathcal{Z}_{t_{i}/n},y_{i}\big\rangle_{\eta}\le a_{i}\right\}\cap\big\{\mathcal{Z}_{n}^{*} \le z(t_{N}/n)\big\}\right]
    \\ &\leq \Phi_{\sigma(\Xi \xi_g)^2 \Sigma}(a_1,\dots, a_N),
    \end{align*}
    where $\Phi_{\sigma(\Xi \xi_g)^2 \Sigma}$ denotes the distribution function of $\mathcal{N}\big(0, \sigma(\Xi \xi_g)^2 \Sigma \big)$. 
    
    \textit{Step 3.} Let us now prove an analogous result for the lower bound. Namely, proceeding as in step 2, we obtain 
    \begin{align*}
            &\mathbb{P}\left[\bigcap_{i=1}^{N}\left\{\sqrt{\lambda_{i}(n)}\big\langle\mathcal{Y}_{t_{i}/n}+\mathcal{Z}_{t_{i}/n},y_{i}\big\rangle_{\eta}\le a_{i}\right\}\cap\big\{\mathcal{Z}_{n}^{*} \le z(t_{N}/n)\big\}\right]
            \\ &\ge\mathbb{P}\left[\bigcap_{i=1}^{N}\left\{\sqrt{\lambda_{i}(n)}\big(\big\langle\mathcal{Y}_{t_{i}/n},y_{i}\big\rangle_{\eta}+z(t_{N}/n)\|y_{i}\|_{\eta}\big)\le a_{i}\right\}\cap\big\{\mathcal{Z}_{n}^{*} \le z(t_{N}/n)\big\}\right]
            \\ &\ge\mathbb{P}\left[\bigcap_{i=1}^{N}\left\{\sqrt{\lambda_{i}(n)}\big(\big\langle\mathcal{Y}_{t_{i}/n}^{\sigma},y_{i}\big\rangle_{\eta}+\overline{\mu}_{i}(n)\big)\le a_{i}\right\}\right]-\mathbb{P}\big[\mathcal{Z}_{n}^{*} > z(t_{N}/n)\big],
    \end{align*}
    where we have set $\overline{\mu}_{i}(n) = \big\langle\mathcal{Y}_{t_{i}/n}^{b},y_{i}\rangle_{\eta} + z(t_{N}/n)\hspace{0.05cm}\|y_{i}\|_{\eta}$ for $i \in\{1,\dots, N\}$. According to \eqref{eq:MarkovineqforMaxCLTlift}, the second term on the right-hand side converges to zero as $n \to \infty$. Moreover, arguing similarly to step 2 (by using Lemma \ref{lemma: Gaussian integral}), we see that the $N$-dimensional Gaussian random vector
    \begin{displaymath}
        \left(\sqrt{\lambda_{i}(n)}\big(\big\langle\mathcal{Y}_{t_{i}/n}^{\sigma},y_{i}\big\rangle_{\eta}+\overline{\mu}_{i}(n)\big),\ i\in\{1,\dots,N\}\right)^{\intercal}
    \end{displaymath}
    converges weakly to $\mathcal{N}\big(0, \sigma(\Xi \xi_g)^2 \Sigma \big)$. In particular, we obtain
    \begin{align*}
        \liminf_{n \to \infty} \mathbb{P}&\left[\bigcap_{i=1}^{N}\left\{\sqrt{\lambda_{i}(n)}\big\langle\mathcal{Y}_{t_{i}/n}+\mathcal{Z}_{t_{i}/n},y_{i}\big\rangle_{\eta}\le a_{i}\right\}\cap\big\{\mathcal{Z}_{n}^{*} \le z(t_{N}/n)\big\}\right]
        \\ &\geq \Phi_{\sigma(\Xi \xi_g)^2 \Sigma}(a_1,\dots, a_N).
    \end{align*}
    Since $(a_{1},\dots,a_{N})^{\intercal}$ is an arbitrary, but fixed continuity point of the desired limiting distribution, the assertion follows by a combination of step~2 and step~3.
\end{proof}

\begin{remark}\label{remark:moreGregularity} 
    Observe that it is sufficient for the argument to have a Hölder-type bound for $g$ as in \eqref{eq:gquasiHölder}, where $0$ is fixed. Moreover, as soon as $\eta_g \ge \eta_*+1$, we get $\chi_g \geq 1/2$, and hence the influence of $\chi_g$ in \eqref{eq:LiftCLTexponentcorridorgeneral} is redundant. Additionally, if $\eta_g > \eta_*+2$, then we may even show that $g$ is Lipschitz continuous. Indeed, as in the proof of Lemma \ref{lemma:liftkernelL2estimates}, $g'$ may be written as $\Xi'S(\cdot)(-\xi_g)$ for $\mu'$ defined via $\frac{\mathrm{d}\mu'}{\mathrm{d}\mu}=x$ with $\eta_{*}'=\eta_*+1$ and $\xi_g\in\mathcal{H}'_{\eta_g-1}$, where $\Xi'$, $\mathcal{H}'_{\eta}$ and $\eta_{*}'$ are defined analogously to $\Xi$, $\mathcal{H}_{\eta}$ and $\eta_*$. Hence, as $\eta_g-1>\eta_{*}'$, boundedness of $g'$ follows from Proposition~\ref{remark: markovian lift}, which completes the argument.
\end{remark}

\begin{remark}\label{remark:LiftCLTtimepoints}
    Note that since we have $0 < t_1 \le \dots \le t_N$, the potentially degenerate case $t_i=t_j$ for some $i,j\in \{1,\dots,N\}$, depending on $y_1,\dots,y_N$, is covered as well, which is of importance for Theorem \ref{theorem:LiftfCLT} below. Moreover, we want to point out that having time points of increasing order is merely motivated by notational convenience for the covariance function in condition~(ii). Indeed, an application of the continuous mapping theorem and the structural stability of the multivariate Gaussian distribution with regards to permutation matrices shows that the above CLT also holds for a potentially unordered collection of positive time points $(t_j)_{j\in\{1,\dots,N\}}$.
\end{remark}

\subsection{A functional CLT for the Markovian Lift}\label{subsection: fCLTLift}

Similarly to the finite-dimensional case considered in Section~\ref{section:smalltimeCLTSVIE}, also for projections of the Markovian lift, the limit covariance matrix $\Sigma$ is often given by an underlying Gaussian process, i.e., it has the form  
\begin{align}\label{eq: 2}
    \Sigma_{ij} = \mathrm{cov}\left(\langle \overline{\mathcal{Y}}_{t_i}, y_i \rangle_{\eta}, \langle \overline{\mathcal{Y}}_{t_j}, y_j \rangle_{\eta} \right),
\end{align}
with the Gaussian process $\overline{\mathcal{Y}}$ on $\mathcal{H}_{\eta}$ given by
\begin{align}\label{eq: 3}
  \overline{\mathcal{Y}}_t = \int_0^{t}S(t-s) \overline{\xi}_K \, \mathrm{d}B_s,\quad t\in\R_+,
\end{align}
where $\overline{\xi}_K \in \mathcal{H}_{-\eta_* - \varepsilon}$ for some $\varepsilon > 0$ small enough. In such a case we can prove a functional CLT for the Markovian lift as stated below.

\begin{theorem}\label{theorem:LiftfCLT}
    Suppose that condition (A) is satisfied, and let $b \in C^{\chi_b}(\R)$ and $\sigma \in C^{\chi_{\sigma}}(\R)$ for some $\chi_b, \chi_{\sigma} \in (0,1]$. Let $K$ have representation \eqref{eq: lift representation}, and let $\mathcal{X}$ be any continuous weak solution of \eqref{eq: markovian lift} on $\mathcal{H}_{\eta_{\mathcal{X}}}$ with $\eta_{\mathcal{X}} \in (\eta_*, 1 - \eta_*)$ and $g = \Xi S(t)\xi_g$, where $\xi_g \in \mathcal{H}_{\eta_g}$ for some $\eta_g > \eta_*$. Fix $N \geq 1$, and let $y_1,\dots, y_N \in \mathcal{H}_{\eta}\setminus\{0\}$ with $\eta\le \eta_{\mathcal{X}}$ satisfy condition (i) from Theorem \ref{theorem:LiftClt}. Finally, suppose that there exists $\overline{\xi}_K \in \mathcal{H}_{-\eta_* - \varepsilon}$ for some $\varepsilon > 0$ such that we have for every $s, t\in(0,T]$ with $s\le t$ and $i,j\in \{1,\dots N\}$:
    \begin{equation}\label{eq:condition2inLiftfCLT}
        \begin{aligned}
            \lim_{n \to \infty} &\sqrt{\lambda_{i}(n)\lambda_{j}(n)} \int_0^{s/n} \langle S((t-s)/n+r)\xi_{K}, y_j\rangle_{\eta}\hspace{0.02cm} \langle S(r)\xi_{K}, y_i\rangle_{\eta}\,\mathrm{d}r \\
            =  &\int_0^{s} \langle S(t-s+r)\overline{\xi}_{K}, y_j\rangle_{\eta}\hspace{0.02cm} \langle S(r)\overline{\xi}_{K}, y_i\rangle_{\eta}\,\mathrm{d}r\\
            = &\hspace{0.1cm}\mathrm{cov}\left(\langle \overline{\mathcal{Y}}_{s}, y_i \rangle_{\eta}, \langle \overline{\mathcal{Y}}_{t}, y_j \rangle_{\eta} \right),
        \end{aligned}
    \end{equation}
    with $\overline{\mathcal{Y}}$ defined by \eqref{eq: 3}. If there exists $\theta > \eta - 1/2$ such that
    \begin{equation}\label{eq:fCLTLiftthetacondition}
        \int_{\R_+}(1+x)^{\theta}\hspace{0.02cm}|y_j(x)|\,\mu(\mathrm{d}x) < \infty 
    \end{equation}
    holds for all $j \in \{1,\dots, N\}$ and
    \begin{equation}\label{eq:LiftfCLTgammacondition}
        \gamma_* < \frac{1}{2} + \theta - \eta,
    \end{equation}
    then, as $n\rightarrow\infty$, we obtain
    \begin{align*}
    &\left( \sqrt{\lambda_1(n)} \langle \mathcal{X}_{t/n} - S(t/n)\xi_g, y_1\rangle_{\eta}, \dots, \sqrt{\lambda_N(n)} \langle \mathcal{X}_{t/n} - S(t/n)\xi_g, y_N\rangle_{\eta} \right)_{t \in [0,T]} 
    \\ &\qquad \qquad \stackrel{d}{\longrightarrow} \sigma\big(\Xi\xi_g\big)\left( \langle \overline{\mathcal{Y}}_t, y_1 \rangle_{\eta}, \dots, \langle \overline{\mathcal{Y}}_t, y_N \rangle_{\eta}\right)_{t \in [0,T]}.
    \end{align*}
\end{theorem}
\begin{proof}
    First, fix $M\in\mathbb{N}$ and $0<t_1<\cdots<t_M\le T$. Considering the sequences $(\widetilde{t}_i)_{i\in\{1,\dots,MN\}}$ and $(\widetilde{y}_i)_{i\in\{1,\dots,MN\}}$ defined by $\widetilde{t}_i=t_{\lceil i/N\rceil}$ and $\widetilde{y}_i=y_{i-\lfloor (i-1)/N\rfloor N}$ in combination with \eqref{eq:condition2inLiftfCLT} and $\overline{\mathcal{Y}}$ being a Gaussian process according to Lemma \ref{lemma: Gaussian integral} shows that condition (ii) from Theorem \ref{theorem:LiftClt} holds with the covariance matrix $\Sigma$ given by
    \begin{displaymath}
        \Sigma = \mathrm{cov}\left(\langle \overline{\mathcal{Y}}_{\widetilde{t}_i}, \widetilde{y}_i \rangle_{\eta}, \langle \overline{\mathcal{Y}}_{\widetilde{t}_j}, \widetilde{y}_j \rangle_{\eta} \right)_{i,j\in\{1,\dots,MN\}}.
    \end{displaymath}
   Therefore, Theorem \ref{theorem:LiftClt} is applicable (see also Remark \ref{remark:LiftCLTtimepoints}). As $M$ and the family of time points $(t_i)_{i\in\{1,\dots,M\}}$ were arbitrary, we have thus shown convergence of finite-dimensional distributions. Hence, it remains to verify that the sequence of the $N$-dimensional random processes 
    \[
        \left( \sqrt{\lambda_1(n)} \langle \mathcal{X}_{t/n} - S(t/n)\xi_g, y_1\rangle_{\eta}, \dots, \sqrt{\lambda_N(n)} \langle \mathcal{X}_{t/n} - S(t/n)\xi_g, y_N\rangle_{\eta} \right)_{t \in [0,T]},
    \]
    $n\in\mathbb{N}$, is tight. To prove the latter, it suffices by \cite[Corollary XIII.1.6]{revuzyor} to prove tightness for each of the $N$ components. For this purpose, let us first obtain some analogous bounds for a collection of auxiliary completely monotone Volterra kernels determined by $y_1,\dots, y_N$. 
    
    \textit{Step 1.} First, assuming that each $y_j$, $j \in \{1,\dots, N\}$, is nonnegative, fix $j$ and let us define the locally finite measure $\nu_j$ on $\R_+$ via $\frac{\mathrm{d}\nu_j}{\mathrm{d}\mu}(x) = y_j(x)\hspace{0.02cm}(1+x)^{\eta}$, and set $K_j(t) = K(\infty)\hspace{0.03cm}y_j(0) + \int_{\R_+}e^{-tx}\,\nu_j(\mathrm{d}x)$. Then,
    due to~\eqref{eq:fCLTLiftthetacondition}, $\nu_j$ satisfies condition~(A)
    from Section~\ref{subsection: LiftHilbertSpace} with constant 
    \begin{equation}\label{eq:LiftfCLTetastarj}
        \eta_*(j) \le \eta - \theta < \frac{1}{2}.
    \end{equation}
    Moreover, we obtain $\langle S(\cdot)\xi_K, y_j\rangle_{\eta}= K_j$ and hence using Lemma \ref{lemma:liftkernelL2estimates} for $\nu_j$ gives
    \[
        \int_0^h |\langle S(r)\xi_K, y_j\rangle_{\eta}|^2\, \mathrm{d}r = \int_0^h K_j(r)^2\, \mathrm{d}r \leq \overline{C} \cdot \begin{cases} h, & \eta_*(j) < 0
        \\ h^{1 - 2\eta_*(j) - \varepsilon}, & 0 \leq \eta_*(j) < \frac{1}{2}, \end{cases}
    \]
    and by the Cauchy-Schwarz inequality also 
    \begin{align*}
          \int_0^h |\langle S(r)\xi_K, y_j\rangle_{\eta}|\, \mathrm{d}r
          \leq \overline{C}^{1/2}\cdot \begin{cases} h, & \eta_*(j) < 0
        \\ h^{1 - \eta_*(j) - \varepsilon/2}, & 0 \leq \eta_*(j) < \frac{1}{2} .\end{cases}
    \end{align*}
    Similarly, we find by the nonnegativity of $y_j$:
    \begin{align*}
        \int_0^t |\langle (S(h + r) - S(r))\xi_K, y_j \rangle_{\eta}|^2\, \mathrm{d}r
        &\leq \int_0^t |K_j(h + r) - K_j(r)|^2\,\mathrm{d}r
        \\ &\leq \widetilde{C} \cdot \begin{cases}h, & \eta_*(j) < 0 
        \\ h^{1-2\eta_{*}(j)-\varepsilon}, & 0 \leq \eta_*(j) < \frac{1}{2}. \end{cases}
    \end{align*}
    
    With these estimates at hand, let us now proceed to prove the desired tightness. Firstly, note that by \eqref{eq: markovian lift} combined with the commutativity of the continuous linear functional $\langle\cdot, y_j\rangle_{\eta}$ with $\mathcal{H}_{\eta}$-valued Bochner and stochastic integrals, or alternatively by a (stochastic) Fubini argument, and $X:=\Xi\mathcal{X}$ defining a continuous solution of \eqref{eq: sve}, as argued already in the proof of Theorem \ref{theorem:LiftClt}, we find
    \begin{align*}
        \langle \mathcal{X}_{t/n} - S(t/n)\xi_g, y_j\rangle_{\eta}
        = \ &\int_0^{t/n}\langle S(t/n - r)\xi_K, y_j \rangle_{\eta}\hspace{0.03cm}b(X_r)\, \mathrm{d}r \\
        &+ \int_0^{t/n} \langle S(t/n - r)\xi_K, y_j \rangle_{\eta}\hspace{0.03cm}\sigma(X_r)\, \mathrm{d}B_r.
    \end{align*}
    Hence, we obtain for $0 \leq s < t \leq T$ and $p\ge 2$:
    \begin{align*}
        &\ \E\left[\left| \langle \mathcal{X}_{t/n} - S(t/n)\xi_g, y_j\rangle_{\eta} - \langle \mathcal{X}_{s/n} - S(s/n)\xi_g, y_j\rangle_{\eta} \right|^p \right] \lesssim I_1 + \dots + I_4,
    \end{align*}
    where the terms $I_1,\dots, I_4$ are given by
    \begin{align*}
        I_1 &= \E\left[ \left(\int_0^{s/n} |\langle (S(t/n - r) - S(s/n - r))\xi_K, y_j \rangle_{\eta}|\hspace{0.03cm} |b(X_r)|\, \mathrm{d}r \right)^p \right],
        \\ I_2 &= \E\left[ \left(\int_{s/n}^{t/n} |\langle S(t/n - r)\xi_K, y_j\rangle_{\eta}|\hspace{0.03cm}|b(X_r)|\, \mathrm{d}r \right)^p \right],
        \\ I_3 &= \E\left[ \left(\int_0^{s/n} |\langle (S(t/n - r) - S(s/n - r))\xi_K, y_j \rangle_{\eta}|^2\hspace{0.03cm} |\sigma(X_r)|^2\, \mathrm{d}r \right)^{p/2}\right],
        \\ I_4 &= \E\left[\left(\int_{s/n}^{t/n} |\langle S(t/n - r)\xi_K, y_j\rangle_{\eta}|^2\hspace{0.03cm}|\sigma(X_r)|^2\, \mathrm{d}r\right)^{p/2} \right].
    \end{align*}
    To bound the latter, we proceed analogously to the proof of Proposition \ref{corollary:zasymptotics} and Lemma~\ref{lemma:momentestimatewithg}. For the first term, we obtain from applying Hölder's inequality twice, Fubini's theorem, $b$ being of linear growth and the moment estimate \eqref{eq:momentestimatewithg} from Lemma \ref{lemma:momentestimatewithg} combined with $\|g\|_{\infty}<\infty$:
    \begin{align*}
        I_1 &\lesssim \left( \int_0^{s/n} |\langle (S(t/n - r) - S(s/n - r))\xi_K, y_j \rangle_{\eta}|\, \mathrm{d}r \right)^{p-1} 
        \\ &\qquad \qquad \cdot \int_0^{s/n} |\langle (S(t/n - r) - S(s/n - r))\xi_K, y_j \rangle_{\eta}| \left(1 + \E[|X_r|^p]\right)\, \mathrm{d}r
        \\ &\lesssim \left( \int_0^{s/n} |\langle (S(t/n - r) - S(s/n - r))\xi_K, y_j \rangle_{\eta}| \,\mathrm{d}r \right)^{p}
        \\ &\lesssim n^{-p/2}\left( \int_0^{s/n}|\langle (S((t-s)/n + r) - S(r))\xi_K, y_j \rangle_{\eta}|^2 \,\mathrm{d}r \right)^{p/2}
        \\ &\lesssim n^{-p/2}\begin{cases}\left(\frac{t-s}{n} \right)^{p/2}, & \eta_*(j) < 0 
        \\ \left( \frac{t-s}{n}\right)^{(p/2)(1-2\eta_{*}(j)-\varepsilon)}, & 0 \leq \eta_*(j) < \frac{1}{2}.\end{cases}
    \end{align*}
    Likewise, we obtain for the second term
    \begin{align*}
        I_2 \lesssim \left( \int_0^{\frac{t-s}{n}} | \langle S(r)\xi_K, y_j \rangle_{\eta}| \,\mathrm{d}r \right)^{p} 
        \lesssim \begin{cases} \left(\frac{t-s}{n}\right)^p, & \eta_*(j) < 0
        \\ \left( \frac{t-s}{n}\right)^{p(1 - \eta_*(j) - \varepsilon/2)}, & 0 \leq \eta_*(j) < \frac{1}{2} .\end{cases}
    \end{align*}
    For $I_3$ we may find an upper bound by
    \begin{align*}
        I_3 &\lesssim \left( \int_0^{s/n} |\langle (S(t/n - r) - S(s/n - r))\xi_K, y_j \rangle_{\eta}|^2 \,\mathrm{d}r \right)^{p/2}
        \\ &\lesssim \begin{cases} \left( \frac{t-s}{n}\right)^{p/2}, & \eta_*(j) < 0 
        \\ \left( \frac{t-s}{n} \right)^{(p/2)(1-2\eta_{*}(j)-\varepsilon)}, & 0 \leq \eta_*(j) < \frac{1}{2}, \end{cases}
    \end{align*}
    while estimating $I_4$ yields
    \begin{align*}
        I_4 \lesssim \left( \int_0^{\frac{t-s}{n}} | \langle S(r)\xi_K, y_j \rangle_{\eta}|^2 \,\mathrm{d}r \right)^{p/2} 
        \lesssim \begin{cases} \left( \frac{t-s}{n}\right)^{p/2}, & \eta_*(j) < 0 
        \\ \left( \frac{t-s}{n} \right)^{(p/2)(1-2\eta_{*}(j)-\varepsilon)}, & 0 \leq \eta_*(j) < \frac{1}{2}. \end{cases}
    \end{align*}
    Collecting all inequalities and combining this with $\sqrt{\lambda_{j}(n)} \lesssim n^{\gamma_*}$ by condition (i) and $|t-s|\le T$ we arrive at
    \begin{equation}\label{eq:LiftfCLTtightnessfinalmoment}
    \begin{aligned}
        &\lambda_j(n)^{p/2}\,\E\left[\left| \langle \mathcal{X}_{t/n} - S(t/n)\xi_g, y_j\rangle_{\eta} - \langle \mathcal{X}_{s/n} - S(s/n)\xi_g, y_j\rangle_{\eta} \right|^p \right]
        \\ &\lesssim \begin{cases} n^{p (\gamma_* - 1/2)}\hspace{0.03cm} (t-s)^{p/2} , & \eta_*(j) < 0
        \\ n^{p(\gamma_* - 1/2+\eta_*(j) + \varepsilon/2)}\hspace{0.03cm}(t-s)^{(p/2)(1-2\eta_*(j) - \varepsilon)}, & 0 \leq \eta_*(j) < \frac{1}{2}.
        \end{cases}
    \end{aligned}
    \end{equation}
    
    \textit{Step 2.} For general $y_j$, write $y_j=y_j^+-y_j^-$, where $y_j^+, y_j^-\in\mathcal{H}_{\eta}$ denote the positive and the negative part of $y_j$, respectively. As the cases with $y_j^+\equiv 0$ or $y_j^-\equiv 0$ are immediate consequences of step 1, we assume in the following that both parts do not vanish. Since \eqref{eq:fCLTLiftthetacondition} simultaneously holds for $y_j^+$ and $y_j^-$ with the same $\theta > \eta - 1/2$, we can in both cases carry out similar arguments as in step 1 for the kernels $\langle S(\cdot)\xi_K, y_j^+\rangle_{\eta}$ and $\langle S(\cdot)\xi_K, y_j^-\rangle_{\eta}$, both being non-degenerate. Hence, by the triangle inequality and \eqref{eq:LiftfCLTtightnessfinalmoment} specified to both cases we obtain with $\varepsilon>0$ sufficiently small:
    \begin{align}\label{eq:LiftfCLTtightnessfinalmomentposuneg}
        \notag &\lambda_j(n)^{p/2}\,\E\left[\left| \langle \mathcal{X}_{t/n} - S(t/n)\xi_g, y_j\rangle_{\eta} - \langle \mathcal{X}_{s/n} - S(s/n)\xi_g, y_j\rangle_{\eta} \right|^p \right]
        \\ \notag &\lesssim \lambda_j(n)^{p/2}\,\E\left[\left| \langle \mathcal{X}_{t/n} - S(t/n)\xi_g, y_j^+\rangle_{\eta} - \langle \mathcal{X}_{s/n} - S(s/n)\xi_g, y_j^+\rangle_{\eta} \right|^p \right]
        \\ \notag &\hspace{0.45cm}+\lambda_j(n)^{p/2}\,\E\left[\left| \langle \mathcal{X}_{t/n} - S(t/n)\xi_g, y_j^-\rangle_{\eta} - \langle \mathcal{X}_{s/n} - S(s/n)\xi_g, y_j^-\rangle_{\eta} \right|^p \right]
        \\ &\lesssim \begin{cases} n^{p (\gamma_* - 1/2)}\hspace{0.03cm} (t-s)^{p/2} , & \overline{\eta}_*(j) < 0
        \\ n^{p(\gamma_* - 1/2+\overline{\eta}_*(j) + \varepsilon/2)}\hspace{0.03cm}(t-s)^{(p/2)(1-2\overline{\eta}_*(j) - \varepsilon)}, & 0 \leq \overline{\eta}_*(j)  < \frac{1}{2},
        \end{cases}
    \end{align}
    where we introduced $\overline{\eta}_*(j):=\max\big\{\eta_*^+(j), \eta_*^-(j)\big\}$ with $\eta_*^+(j)$ and $\eta_*^-(j)$ being defined analogously to $\eta_*$ for the kernels $\langle S(\cdot)\xi_K, y_j^+\rangle_{\eta}$ and $\langle S(\cdot)\xi_K, y_j^-\rangle_{\eta}$, respectively. Combining both $\eta_*^+(j)\le \eta-\theta$ and $\eta_*^-(j)\le \eta-\theta$ with \eqref{eq:LiftfCLTgammacondition} shows
    \begin{displaymath}
        \gamma_*<\tfrac{1}{2}+\min\left\{-\eta_*^+(j), -\eta_*^-(j)\right\}=\tfrac{1}{2}-\max\left\{\eta_*^+(j), \eta_*^-(j)\right\}=\tfrac{1}{2}-\overline{\eta}_*(j).
    \end{displaymath}
    Hence, noting that we have in general $\gamma_*\le 1/2$ by Lemma \ref{lemma:liftkernelL2estimates} (see also \eqref{eq:LiftfCLTlambdaupperbound} below) and selecting $\varepsilon>0$ sufficiently small yields uniform boundedness in $n\in\mathbb{N}$. Therefore, combining this with $\mathcal{X}_{0/n}-S(0/n)\xi_g\equiv 0$ for all $n\in\mathbb{N}$ allows to apply Kolmogorov's tightness criterion (see e.g.\ \cite[Theorem XIII.1.8]{revuzyor}) for $p$ sufficiently large. Thus, we have shown that the sequence of one-dimensional stochastic processes
    \begin{displaymath}
        \left(\left( \sqrt{\lambda_j(n)} \langle \mathcal{X}_{t/n} - S(t/n)\xi_g, y_j\rangle_{\eta} \right)_{t \in [0,T]}\right)_{n\in\mathbb{N}}
    \end{displaymath}
    is tight. Since this holds for every $j\in\{1,\dots,N\}$, we have verified tightness also for the sequence of the entire processes, which completes the proof.
\end{proof}

The proof reveals that the functional central limit theorem is valid whenever condition~(i) of Theorem~\ref{theorem:LiftClt} and \eqref{eq:condition2inLiftfCLT} hold and we may verify a uniform bound for \eqref{eq:LiftfCLTtightnessfinalmoment}, when $y_j\ge 0$ or $y_j\le 0$, and \eqref{eq:LiftfCLTtightnessfinalmomentposuneg} in the general case. The latter requires an estimate on the parameter $\gamma_*$ which we have obtained under conditions \eqref{eq:fCLTLiftthetacondition} and \eqref{eq:LiftfCLTgammacondition}. Independently of the latter conditions, when, depending on the case, $\eta_*(j) < 0$ or $\overline{\eta}_*(j) < 0$ holds, we always get $\gamma_* = 1/2$ and the above corollary is applicable as illustrated in the example below for the classical projection case.

\begin{example}\label{example:LiftfCLTprojectionetastarsmaller0}
    For $\eta_*<0$ and an $\mathcal{H}_{\eta_\mathcal{X}}$-valued solution to \eqref{eq: markovian lift} with $\eta_\mathcal{X}\in (\eta_*, 1-\eta_*)$ consider $\eta\in (\eta_{*},0\wedge \eta_\mathcal{X})$ and choose $N=1$ and $y_{1}=w_{\eta}\ge 0$, which implies $\langle S(\cdot)\xi_{K}, y_1\rangle_{\eta}=~K$ and hence $\nu_1=\mu$ in step 1 of the proof of Theorem \ref{theorem:LiftfCLT}. Lemma~\ref{lemma:liftkernelL2estimates} gives $\gamma_{*}=1/2$ and hence \eqref{eq:LiftCLTexponentcorridorgeneral} is trivially satisfied since $\eta<0$. This shows that condition (i) holds. With regards to \eqref{eq:condition2inLiftfCLT}, we observe that by boundedness from above and below (see Lemma \ref{lemma:liftkernelL2estimates}) and the continuity of $K$ on $\R_{+}^*$ we obtain $K\sim K(0):=\lim_{t\searrow0}K(t)>0$. Therefore, it follows from Example \ref{example:limitingkernelmorethanRL} that condition (ii) of Theorem \ref{theorem:varyinginitialCLT} holds for $\overline{K}\equiv 1$. Hence, choosing $\overline{\xi}_{K}=\mathbbm{1}_{\{0\}}\in\mathcal{H}_{\eta}$ proves that also \eqref{eq:condition2inLiftfCLT} is satisfied. Moreover, combining $\eta\in (\eta_{*},0)$ with $\gamma_*=1/2$ shows that \eqref{eq:fCLTLiftthetacondition} and \eqref{eq:LiftfCLTgammacondition} are satisfied for $\theta=0$. Hence, Theorem \ref{theorem:LiftfCLT} proves a functional CLT for the process $(X_t)_{t\in[0,T]}:=(\Xi\mathcal{X})_{t\in[0,T]}$ solving the SVIE \eqref{eq: sve} in the spirit of Corollary \ref{corollary:functionalCLT}, where the limit process is given by scaled Brownian motion, i.e.\ $(\sigma(g(0))B_{t})_{t\in [0, T]}$. 
\end{example}

The situation is more delicate when $\overline{\eta}_*(j) \in [0,1/2)$. Indeed, to illustrate possible bounds on $\gamma_*$, let us assume without loss of generality that $y_j\ge 0$ and recall that we have for every $j\in\{1,\dots,N\}$:
\begin{align}\label{eq:LiftfCLTlambdaupperbound}
        \sqrt{\lambda_{j}(n)} = \left( \int_0^{1/n} \langle S(r)\xi_{K}, y_j\rangle_{\eta}^2\, \mathrm{d}r \right)^{-1/2}\lesssim \big(\nu_{j}((0,n])\vee 1\big)^{-1} \sqrt{n},
\end{align}
which follows immediately from $K_j=\langle S(\cdot)\xi_K, y_j\rangle_{\eta}$ being a completely monotone kernel satisfying condition~(A) and the lower bound obtained in Lemma \ref{lemma:liftkernelL2estimates} for $h=1/n$. Thus, sharp bounds on $\gamma_*$ are closely related to the asymptotic behavior of $\nu_j((0,n])$ as $n \to \infty$. However, in view of Example \ref{example:kernelestimeslowercannotimproved}, we need to compute the order of $\nu_j((0,n])$ in a case-by-case study. Below we carry out such an analysis for the case of Riemann-Liouville kernels allowing us to deduce a functional CLT for $\Xi\mathcal{X}$ by relaxing \eqref{eq:LiftfCLTgammacondition} an showing that \eqref{eq:LiftfCLTtightnessfinalmoment} is still uniformly bounded in~$n$.

\begin{example}\label{example:LiftfCLTprojectionRL}
    Consider the Riemann-Liouville kernel $K(t)=t^{H-1/2}$, $t\in\R_+$, and the corresponding measure $\mu$ as given in Example \ref{example:bernsteinmeasures}~(a). Then $\eta_{*}=1/2-H$ and $\gamma_{*}=H<1/2$ by Example \ref{example:limitingkernelfractional} for these kernels. Given a $\mathcal{H}_{\eta_\mathcal{X}}$-valued solution to \eqref{eq: markovian lift} with $\eta_\mathcal{X}\in (\eta_*, 1-\eta_*)$, we may select 
        \begin{displaymath}
            \eta=\eta_*+\varepsilon \quad \mbox{with}\quad \varepsilon\in \big(0, \min\{2H\chi_{\sigma}, 2\chi_g\chi_{\sigma}, \eta_\mathcal{X}-\eta_* \}\big),
        \end{displaymath} 
        $N=1$ and $y_{1}=w_{\eta}$ so that $\langle S(\cdot)\xi_K, y_1\rangle_{\eta}=K$. Hence, we obtain \eqref{eq:LiftCLTexponentcorridorgeneral}, which verifies condition (i). Condition \eqref{eq:condition2inLiftfCLT} follows from Example \ref{example:limitingkernelfractional} with $\overline{\xi}_{K}=\sqrt{2H}\xi_{K}$. With regards to $\theta > \eta - 1/2$, \eqref{eq:fCLTLiftthetacondition} and \eqref{eq:LiftfCLTgammacondition}, the latter combined with $\eta>\eta_*$ necessarily implies $\theta>\eta-\eta_*>0$. Inserting this into \eqref{eq:fCLTLiftthetacondition} with $y_{1}=w_{\eta}$ shows that the integral there cannot be finite since $\int_{\R_+} (1+x)^{-\eta_*}\,\mu(\mathrm{d}x)=\infty$ holds in this case. However, as we know the exact bounds for the Riemann-Liouville kernel and $\langle S(\cdot)\xi_K, y_1\rangle_{\eta}=K$, one can check directly that \eqref{eq:LiftfCLTtightnessfinalmoment} is uniformly bounded in $n$. Here it is crucial that $\mu((0,x])=C\hspace{0.03cm}x^{1/2-H}=C\hspace{0.03cm}x^{\eta_*}$ holds for every $x\in\R_+^*$. Hence, also here we obtain a functional CLT for the process $X_{t\in[0,T]}:=(\Xi\mathcal{X})_{t\in[0,T]}$ solving the SVIE \eqref{eq: sve}. 
\end{example}

\begin{remark}
    Note that in the last two examples, as soon as we have a continuous solution~$X$ to \eqref{eq: sve}, a continuous, $\mathcal{H}_{\eta_\mathcal{X}}$-valued lift $\mathcal{X}$ with $\Xi\mathcal{X}=X$ exists due to Theorem~\ref{thm: Markovian lift Markov property}. Hence, by the above arguments, a functional CLT holds also for the original process~$X$.
\end{remark}

\subsection{Examples}\label{subsection: ExamplesforLiftCLT}

In the final part of this section, we discuss a collection of examples where Theorem \ref{theorem:LiftClt} and its fCLT version can be applied. While the general formulation allows for a vast class of possible functions $y_1,\dots, y_N$, below we focus on a particular subclass for which all conditions can be verified by similar arguments to the examples given in Subsection \ref{subsection:CLTKernelExamples}. Namely, we consider functions of the form
\[
    y(x) =  (1+x)^{-\eta} \,\frac{\mathrm{d}\overline{\nu}}{\mathrm{d}\overline{\mu}_{K}}(x), \quad x\in\R_+,
\]
where $\overline{\nu}$ is another Bernstein measure on $\R_+$ (see Remark \ref{remark:mudefinitions}) that is absolutely continuous with respect to $\overline{\mu}_{K}$. Let $K_{\overline{\nu}}(t) = \int_{\R_+}e^{-tx} \,\overline{\nu}(\mathrm{d}x)$, then it follows $\langle S(t)\xi_K, y \rangle_{\eta} = K_{\overline{\nu}}(t)$, provided $y\in\mathcal{H}_{\eta}$, and hence conditions (i) and (ii) in Theorem \ref{theorem:LiftClt} reduce for our completely monotone kernel $K_{\overline{\nu}}$ to the classical finite-dimensional framework discussed in Section~\ref{section:smalltimeCLTSVIE}. However, here we allow in some sense more generality via the function $g$ and each component having potentially a different kernel (see \eqref{eq:LiftCLTcomponentprocesses}), albeit at the cost of restricting ourselves to completely monotone kernels. 

\begin{corollary}\label{cor: CLT Markov lift}
    Let $K$ be completely monotone with representation \eqref{eq: lift representation} and $\mu$ satisfying condition~(A) and let $b \in C^{\chi_b}(\R)$ and $\sigma \in C^{\chi_{\sigma}}(\R)$ for some $\chi_b, \chi_{\sigma} \in (0,1]$. Consider a continuous weak solution $X$ of \eqref{eq: sve}, where $g \in \mathcal{G}_{\eta_g}$ for some $\eta_g > \eta_*$. Now consider a family of potentially different completely monotone kernels $K_{1}, \dots, K_N$, whose Bernstein measures on~$\R_+$ are denoted by $\overline{\nu}_1,\dots, \overline{\nu}_N$ and satisfy $\overline{\nu}_j\ll\overline{\mu}_{K}$, $\frac{\mathrm{d}\overline{\nu}_j}{\mathrm{d}\overline{\mu}_{K}}\in L^2([0,1]; \overline{\mu}|_{[0,1]})$ and
    \begin{equation}\label{eq:LiftCLTcorollarydensityestimate}
        \frac{\mathrm{d}\overline{\nu}_j}{\mathrm{d}\overline{\mu}_{K}}(x) \leq C(1+x)^{\rho_j}, \qquad \forall x\in [1, \infty),\ \forall j\in\{1,\dots,N\},
    \end{equation}
    holds for some constants $C > 0$ and $\rho_j < \frac{1}{2} - \eta_*$. Moreover, suppose that the following set of conditions is satisfied for fixed $0 \le t_1 \le \dots \le t_N$:
    \begin{enumerate}
        \item[(i)] There exists $\gamma_* > 0$ such that
        \[
            \lambda_j(n) = \left( \int_0^{1/n} K_j(r)^2\, \mathrm{d}r \right)^{-1} \leq \overline{C} n^{2\gamma_*}
        \]
        holds for all $j\in\{1,\dots,N\}$ and some constant $\overline{C} > 0$, and, defining $\rho:=\max_{j\in\{1,\dots,N\}}\rho_{j}$:
        \begin{align}\label{eq: 4}
        \gamma_* < \frac{1}{2}- (\eta_* + \rho)^+ + \min\left\{\left( \frac{1}{2} - \eta_*^+ \right), \chi_g\right\} \chi_{\sigma}.
        \end{align}
        \item[(ii)] There exists a symmetric and positive semidefinite $N \times N$-matrix $\Sigma$ such that for all $i,j \in\{1,\dots, N\}$ with $i \leq j$:
        \[
            \lim_{n \to \infty}\sqrt{\lambda_i(n)\hspace{0.02cm}\lambda_j(n)}\int_0^{t_i/n} K_j\left( \frac{t_j - t_i}{n} + r\right)K_i(r)\, \mathrm{d}r = \Sigma_{ij}.
        \]
    \end{enumerate}
    Then, denoting by $\overline{X} = \big(\overline{X}^1,\dots, \overline{X}^N\big)^{\intercal}$ the continuous $N$-dimensional Volterra Ito-process with 
    \begin{equation}\label{eq:LiftCLTcomponentprocesses}
        \overline{X}^j_t =\int_0^t K_j(t-s)\hspace{0.02cm}b(X_s)\,\mathrm{d}s + \int_0^t K_j(t-s)\hspace{0.02cm}\sigma(X_s)\,\mathrm{d}B_s,\quad t\in\R_+,
    \end{equation}
    we conclude as $n \to \infty$:
    \begin{equation}\label{eq:LiftCLTcorollarydifferentkernels}
        \left( \sqrt{\lambda_j(n)}\, \overline{X}_{t_j/n}^j \right)_{j = 1,\dots, N} \stackrel{d}{\longrightarrow} \mathcal{N}\left(0, \sigma(g(0))^2 \hspace{0.02cm}\Sigma \right).
    \end{equation}
\end{corollary}
\begin{proof}
    Let $X$ be a continuous weak solution to \eqref{eq: sve} and note that since $\rho_j < 1/2 - \eta_*$ for every $j\in\{1,\dots,N\}$, the interval $(\eta_* + 2\rho, 1-\eta_*)$ is non-degenerate. Hence, according to our assumptions and Theorem \ref{thm: Markovian lift Markov property}, there exists a continuous solution $\mathcal{X}$ of \eqref{eq: markovian lift} on $\mathcal{H}_{\eta_{\mathcal{X}}}$ with $\eta_{\mathcal{X}} \in (\eta_*+2\rho^+, 1 - \eta_*)$ and $\Xi\mathcal{X}=X$. Moreover, observe that for every $\eta>\eta_*+2\rho$, which shall be chosen below, we may define $y_j:\R_+\rightarrow\R_+$ via
    \begin{displaymath}
        y_j(x) = (1+x)^{-\eta}\,\frac{\mathrm{d}\overline{\nu}_j}{\mathrm{d}\overline{\mu}_{K}}(x) = \frac{\overline{\nu}_j(\{0\})}{K(\infty)}\mathbbm{1}_{\R_+^*}(K(\infty))\mathbbm{1}_{\{0\}}(x) + (1+x)^{-\eta}\, \frac{\mathrm{d}\nu_j}{\mathrm{d}\mu}(x),
    \end{displaymath}
    allowing us to conclude $y_j \in \mathcal{H}_{\eta}$ by $\frac{\mathrm{d}\overline{\nu}_j}{\mathrm{d}\overline{\mu}_{K}}\in L^2([0,1]; \overline{\mu}|_{[0,1]})$, \eqref{eq:LiftCLTcorollarydensityestimate}, $\eta-2\rho_j>\eta_*$ and
    \begin{align*}
         \|y_j\|_{\eta}^2 &=\int_{\R_+} (1+x)^{-\eta}\left(\frac{\mathrm{d}\overline{\nu}_j}{\mathrm{d}\overline{\mu}_{K}}(x)\right)^2\overline{\mu}(\mathrm{d}x)\\
        &\lesssim \left\|\frac{\mathrm{d}\overline{\nu}_j}{\mathrm{d}\overline{\mu}_{K}}\right\|_{L^2([0,1];\overline{\mu}|_{[0,1]})}^2+\int_1^\infty (1+x)^{2\rho_j-\eta}\,\mu(\mathrm{d}x) <\infty,
    \end{align*}
    and, therefore, $\langle S(t)\xi_K, y_j \rangle_{\eta} = K_j(t)$ since $\overline{\nu}_j(\{0\})=K_j(\infty)$ for each $j\in\{1,\dots,N\}$. This shows that the conditions (i) and (ii) here are equivalent to the conditions (i) and~(ii) given in Theorem \ref{theorem:LiftClt}, provided there exists an admissible $\eta$ for \eqref{eq:LiftCLTexponentcorridorgeneral}. Indeed, \eqref{eq:LiftCLTexponentcorridorgeneral} follows from \eqref{eq: 4}
    by choosing~$\eta\in (\eta_* + 2\rho, \eta_{\mathcal{X}}]$ sufficiently close to $\eta_* + 2\rho$. Finally, we observe $\langle\mathcal{X}_t-S(t)\xi_g, y_j\rangle_{\eta}=\overline{X}^j_t$ for every $t\in\R_+$ and $j\in\{1,\dots,N\}$ by $y_j \in \mathcal{H}_{\eta}$ and the commutativity of continuous linear functionals with $\mathcal{H}_{\eta}$-valued Bochner and stochastic integrals. Hence, \eqref{eq:LiftCLTcorollarydifferentkernels} follows from Theorem~\ref{theorem:LiftClt}.
\end{proof}

The assumptions of this corollary are satisfied in a number of different cases as illustrated below. First, we discuss a connection to the CLT framework from Section~\ref{section:smalltimeCLTSVIE}.

\begin{example}
    Taking $\nu_j = \mu$ for every $j\in\{1,\dots,N\}$ corresponds to $y_j = w_{\eta}$, $K_j = K$ and $\rho_j = 0$. Note that by condition~(A) we always have $\overline{\mu}([0,1])<\infty$. Hence, recalling $\gamma=1/2-(\eta_*+\varepsilon)^+$ for any arbitrarily small $\varepsilon>0$ by Lemma \ref{lemma:liftkernelL2estimates} with the notation from Assumption \ref{assumption:coefficientassumptions}, we observe that Corollary \ref{cor: CLT Markov lift} for $g\equiv x_0$ reduces to a form very similar to Theorem \ref{theorem:varyinginitialCLT} for fixed $0 < t_1 < \dots < t_N$ and completely monotone kernels. 
\end{example}
The next example showcases, how \eqref{eq: 4} captures the regularization effect that may occur in the transformation of the Bernstein measures for the special case of the naturally very regular kernel shifts. This is in particular relevant for the case $\eta_*\in[0,1/2)$.
\begin{example}\label{example:LiftCLTforgeneralShifts}
    Constructing for $\varepsilon_1,\dots, \varepsilon_N > 0$ the Bernstein measures on $\R_+$ via $\overline{\nu}_j(\mathrm{d}x) = e^{-\varepsilon_j x}\,\overline{\mu}_{K}(\mathrm{d}x)$ (see also Example \ref{example:bernsteinmeasures} (iv)), we may select $\rho_j = -\eta_{*}^+$ for every $j\in\{1,\dots,N\}$ by the exponential decay of the densities. Then we can conclude $K_j = K(\cdot + \varepsilon_j)$ and $\lambda_j(n) \sim n$, whence it follows in particular $\gamma_* = 1/2$. Hence, it is an immediate consequence of $\rho=-\eta_{*}^+$ that \eqref{eq: 4} is always satisfied, even when the original kernel is not regular. Moreover, $\frac{\mathrm{d}\overline{\nu}_j}{\mathrm{d}\overline{\mu}_{K}}\in L^2([0,1]; \overline{\mu}|_{[0,1]})$ follows directly from the boundedness of the density in combination with $\overline{\mu}([0,1])<\infty$. Finally, by explicit computation and $K_j(t)\sim K(\varepsilon_j)$ we obtain that condition (ii) of Corollary \ref{cor: CLT Markov lift} holds with
    \[
        \Sigma_{ij} = (t_i \wedge t_j)=\mathrm{cov}\big(B_{t_i}, B_{t_j}\big),
    \]
    whence by Corollary \ref{cor: CLT Markov lift} the limiting distributions here coincide with the finite-dimensional distributions of the underlying Brownian motion. Consequently, this example is a generalization of Example \ref{example:limitingkernelmorethanRL} for kernel regularization via shifts, where one is allowed to choose a different shift parameter in every component.
\end{example}

Finally, let us consider as an example a family of fractional kernels of Riemann-Liouville type with different parameters.

\begin{example}
    Motivated by Example \ref{example:bernsteinmeasures} (a), take 
    \begin{displaymath}
        \overline{\mu}_{K}(\mathrm{d}x) = \frac{x^{-\alpha}}{\Gamma(\alpha)\Gamma(1-\alpha)} \hspace{0.03cm}\mathbbm{1}_{\R_+^*}(x)\,\mathrm{d}x \ \ \ \mbox{and} \ \ \ \overline{\nu}_j(\mathrm{d}x) = \frac{x^{-\beta_j}}{\Gamma(\beta_j)\Gamma(1-\beta_j)}\hspace{0.03cm}\mathbbm{1}_{\R_+^*}(x)\,\mathrm{d}x,
    \end{displaymath}
    with parameters satisfying $\alpha\in (1/2,1)$ and $\beta_j \in \big(1/2,(\alpha+1)/2\big)$ for every $j\in\{1,\dots,N\}$. Hence, we obtain $K_j(t) =\Gamma(\beta_j)^{-1}\hspace{0.03cm}t^{\beta_j - 1}$, $\gamma_* = \max_{j} \beta_j - 1/2$ and $\lambda_j(n) = \Gamma(\beta_j)^2\hspace{0.03cm}(2\beta_j - 1)\hspace{0.03cm}n^{2\beta_j - 1}$. With regards to \eqref{eq:LiftCLTcorollarydensityestimate}, observe that for every $j\in\{1,\dots,N\}$ we get $\rho_j = \alpha - \beta_j$ by $[1,\infty)\ni x\mapsto (1+x^{-1})^{-1}\in [1/2, 1]$ as well as
    \begin{displaymath}
        \left\|\frac{\mathrm{d}\overline{\nu}_j}{\mathrm{d}\overline{\mu}_K}\right\|_{L^2([0,1];\overline{\mu}|_{[0,1]})}^2\lesssim \int_0^1 x^{2(\alpha-\beta_j)}x^{-\alpha}\,\mathrm{d}x=\int_0^1 x^{\alpha-2\beta_j}\,\mathrm{d}x<\infty,
    \end{displaymath}
    since we have $\alpha-2\beta_j>-1$ by assumption. Moreover, a short computation, which is very similar to Example \ref{example:limitingkernelfractional}, gives with regards to condition (ii):
    \[
        \Sigma_{ij} = \sqrt{(2\beta_j-1)\hspace{0.03cm}(2\beta_i-1)}\int_{0}^{t_i\wedge t_j} (t_j-r)^{\beta_j -1} \hspace{0.03cm}(t_i - r)^{\beta_i -1}\, \mathrm{d}r.
    \]
    In particular, translating \eqref{eq: 4} into the current example, we observe that Corollary \ref{cor: CLT Markov lift} is applicable whenever
    \[
        \max_{j\in\{1,\dots,N\}} \beta_j < \min_{j\in\{1,\dots,N\}}\beta_j + \min\left\{\left( \alpha - \frac{1}{2}\right), \chi_g\right\}\chi_{\sigma}.
    \]
\end{example}

For applications of the functional CLT framework developed in Theorem~\ref{theorem:LiftfCLT} we refer to the motivating Examples \ref{example:LiftfCLTprojectionetastarsmaller0} and \ref{example:LiftfCLTprojectionRL}, where we have shown for the regular as well as the Riemann-Liouville case that a functional CLT for the original process can be obtained by an application of the classical projection operator $\Xi$ to the Lift $\mathcal{X}$. Inspired by Corollary~\ref{cor: CLT Markov lift}, one could, of course, also take the above one step further by studying the joint distribution of the lift under two different functionals transforming the original Bernstein measure $\mu_{K}$, i.e.\ a functional CLT for $\big(\overline{X}^1, \overline{X}^2\big)^{\intercal}$ on $[0,T]$, where each $K_j$, $j\in\{1,2\}$, corresponds, for example, to a possibly different shift of the original kernel~$K$ (see also Example \ref{example:LiftCLTforgeneralShifts}).

\vspace{0.3cm}
\noindent{\bf Acknowledgement.} We thank Stefano De Marco and Masaaki Fukasawa for helpful comments.

\begin{appendices}

\section{Auxiliary moment bound}

In this section, we prove a moment bound for continuous solutions of \eqref{eq: sve}, which holds in particular in the small-time regime. 

\begin{lemma}\label{lemma:momentestimatewithg}
    Suppose that $b,\sigma$ are continuous with linear growth, $g: [0,T] \longrightarrow \R$ is continuous and bounded, and $K \in L^2([0,T])$ satisfies
    \begin{equation}\label{eq:kernelassumptionmomentAppendix}
        \int_0^t K(s)^2\, \mathrm{d}s \leq C t^{2\gamma}, \qquad t \in (0,T],
    \end{equation}
    for some constants $C, \gamma > 0$ and a time horizon $T>0$. Let $X$ be any continuous weak solution of~\eqref{eq: sve}. Then, for each $p \geq 2$, it holds that
    \begin{equation}\label{eq:momentestimatewithg}
        \mathbb{E}[|X_t - g(t)|^p] \leq C_p\,(1+ \|g\|_{\infty}^p)\,t^{p\gamma}, \quad t\in [0,T],
    \end{equation}
    where $C_p > 0$ is some constant.
\end{lemma}
\begin{proof}
    First, since $b,\sigma$ have linear growth, it follows from a minor modification of the proof of \cite[Lemma 3.1]{AbiJaLaPu19}, where we bound $g$ by its supremum norm (see also \cite[Lemma 3.4]{PROMEL2023291}), that 
    \begin{align}\label{eq: integrability}
        \sup_{t \in [0,T]} \E[|X_t|^p] < \infty, \qquad \forall p \in [1,\infty).
    \end{align}
    Applying the BDG inequality, then Jensen's inequality, and finally 
    \[
        (1 + |X_s|)^p \lesssim \left(1 + |g(s)|\right)^p + |X_s - g(s)|^p
        \lesssim (1 + \|g\|_{\infty})^p + |X_s - g(s)|^p,
    \]
    we find 
    \begin{align*}
        \mathbb{E}\left[ |X_{t}-g(t)|^{p} \right] 
        &\lesssim \E\left[ \left( \int_{0}^{t} |K(t-s)|\,( 1+|X_s| ) \, \mathrm{d}s\right)^{p}\right] 
        + \E\left[ \left( \int_{0}^{t}K(t-s)^{2} \,(1+|X_s|)^2\, \mathrm{d}s \right)^{\frac{p}{2}}\right]
        \\ &\lesssim \left(\int_0^t |K(s)|\, \mathrm{d}s \right)^{p-1} \int_0^t |K(t-s)|\,\E\left[(1+|X_s|)^p \right]\, \mathrm{d}s
        \\ &\qquad + \left(\int_0^t |K(s)|^2\, \mathrm{d}s \right)^{\frac{p}{2}-1} \int_0^t |K(t-s)|^2\,\E\left[(1+|X_s|)^p \right]\, \mathrm{d}s
        \\ &\lesssim \left(\int_0^t |K(s)|\, \mathrm{d}s \right)^{p} \left(1 + \|g\|_{\infty}^p\right)
        + \left(\int_0^t |K(s)|^2\, \mathrm{d}s \right)^{\frac{p}{2}}\left(1 + \|g\|_{\infty}^p\right)
        \\ &\qquad + \left(\int_0^t |K(s)|\, \mathrm{d}s \right)^{p-1} \int_0^t |K(t-s)|\,\E\left[|X_s-g(s)|^p \right]\, \mathrm{d}s
        \\ &\qquad + \left(\int_0^t |K(s)|^2\, \mathrm{d}s \right)^{\frac{p}{2}-1} \int_0^t |K(t-s)|^2\,\E\left[|X_s-g(s)|^p \right]\, \mathrm{d}s.
    \end{align*}
    Hence, the estimate $\int_0^t |K(s)|\, \mathrm{d}s \leq t^{1/2}\left( \int_0^t K(s)^2\, \mathrm{d}s\right)^{1/2} \lesssim t^{\gamma + \frac{1}{2}}$, following from H\"older's inequality, combined with \eqref{eq:kernelassumptionmomentAppendix} gives 
    \begin{align*}
        \mathbb{E}\left[ |X_{t}-g(t)|^{p} \right] 
        &\lesssim t^{p \gamma}\left(1 + \|g\|_{\infty}^p\right) - \int_0^t H(t,s)\,\E\left[|X_s-g(s)|^p \right]\, \mathrm{d}s,
    \end{align*}
    where we have set for $s,t\in [0,T]$:
    \[
        H(t,s) = -\left(t^{(p-1)(\gamma + \frac{1}{2})}\,|K(t-s)| + t^{(p-2)\gamma}\,|K(t-s)|^2\right)\1_{\{s \leq t\}} \leq 0.
    \]
    Due to $K\in L_{\mathrm{loc}}^{2}(\R_{+})$ and $p\ge 2$ it follows that $\sup_{t \in (0,T]}\int_0^t |H(t,s)|\, \mathrm{d}s < \infty$. Hence, using \cite[Proposition 9.2.7~(i)]{grippenberg}, we find that $H$ is a Volterra kernel of $L^\infty$-type in the sense of \cite[Definition 9.2.2]{grippenberg}. Moreover, by \eqref{eq:kernelL2order} and \cite[Corollary 9.3.14]{grippenberg} there exists a resolvent of $L^\infty$-type $R$ associated with $H$. Since $H \leq 0$, it follows from \cite[Proposition 9.8.1]{grippenberg} that also $R \leq 0$. By the generalized Volterra Gronwall inequality given in \cite[Lemma 9.8.2]{grippenberg}, we find 
    \begin{align*}
        \mathbb{E}\left[ |X_{t}-g(t)|^{p} \right] 
        &\leq t^{p\gamma}\left(1 + \|g\|_{\infty}^p\right) - \int_0^t R(t,s)\,s^{p\gamma}\left(1 + \|g\|_{\infty}^p\right)\, \mathrm{d}s
        \\ &\leq t^{p \gamma} \left( 1+ \sup_{t \in (0,T]}\int_0^t |R(t,s)|\, \mathrm{d}s \right)(1+\|g\|_{\infty}^p),
    \end{align*}
    which proves the first assertion as the middle part is finite due to $R$ being of $L^\infty$-type.
\end{proof}

Our bound is very similar to \cite[Lemma 2.4]{AbiJaLaPu19} when $s=0$. However, in our case, we do not require \eqref{eq:kernelincrementL2} and additionally obtain a stronger bound on the right-hand side as it can be seen for the Riemann-Liouville kernel $K(t) = t^{H-1/2}$ when $H > 1/2$, for which we merely have $\overline{\gamma}=1/2$.

\section{Regularization for Markovian lifts}

Similarly to the classical theory of SPDEs, also here we can use regularizing properties of the analytic semigroup $(S(t))_{t \geq 0}$ to prove additional regularity for the Markovian lift $\mathcal{X}$ in the spaces $\mathcal{H}_{\eta}$ with $\eta \in (\eta_*, 1 - \eta_*)$. The latter allows us, in particular, to conclude that $\Xi\mathcal{X}$ is a solution to the associated SVIE~\eqref{eq: sve}. 

\begin{lemma}\label{lemma: regularization lift}
    Suppose that condition (A) holds. Fix $\eta < 1 - \eta_*$ and let $g = \Xi S(\cdot)\xi_g \in \mathcal{G}_{\eta_g}$ for some $\eta_g > \eta_*$. Suppose that $K = \Xi S(\cdot)\xi_K$ is given as in \eqref{eq: lift representation}, and let $X$ be a continuous weak solution of \eqref{eq: sve} with the coefficients $b$, $\sigma$ being continuous and of linear growth. Then there exists a 
    solution~$\mathcal{X}$ to~\eqref{eq: markovian lift}, i.e.\ the corresponding Markovian lift. Moreover, it follows that there exists a version with
    \[
        \mathcal{X} - S(\cdot)\xi_g \in L^2(\Omega, \P; C([0,T]; \mathcal{H}_{\eta})), \qquad \forall T > 0.
    \]
    In particular, since $\eta_* < 1/2$ we may choose $\eta \in (\eta_*, 1 - \eta_*)$ and hence $\Xi$ is well-defined on $\mathcal{X}$ and it holds that $\Xi \mathcal{X} = X$.
\end{lemma}
\begin{proof}
 Firstly, since $\eta < 1-\eta_*$, we find $\delta \in (0,1)$ small enough such that $\eta + \delta < 1-\eta_*$. Defining $\eta' = \eta - (1-\delta)$, it follows $\eta' < \eta$, $\eta' < - \eta_*$ and $\eta - \eta' = 1-\delta < 1$. In particular, we have $\xi_K \in \mathcal{H}_{\eta'}$ by condition (A) and hence using \eqref{eq:liftsemigroupoperatornorm} for $\eta' < \eta$ gives $S(t)\xi_K \in \mathcal{H}_{\eta}$ such that we have for every $T\in\R_+^*$:
    \begin{align}\label{eq: good bound on SK}
        \|S(t)\xi_K\|_{\eta} \leq C_{T}\hspace{0.03cm}\kappa(1-\delta)\hspace{0.03cm} t^{-\frac{1-\delta}{2}}, \qquad \forall t \in (0,T],
    \end{align}
    which also shows that $\|S(\cdot)\xi_K\|_{\eta} \in L_{\mathrm{loc}}^2(\R_+)$. For every $t\in\R_+$ define
    \begin{align}\label{eq: appendix lift constructed solution}
 \mathcal{X}_t = S(t)\xi_g + \int_0^t S(t-s)\xi_K b(X_s)\, \mathrm{d}s + \int_0^t S(t-s)\xi_K \sigma(X_s)\, \mathrm{d}B_s,
\end{align}
    which is well-defined and $\mathcal{X}_{t}\in\mathcal{H}_{\overline{\eta}}$ a.s. for every $\overline{\eta}<1-\eta_*$ and $t>0$. Indeed, notice that by $\xi_g\in\mathcal{H}_{\eta_g}$, \eqref{eq:liftsemigroupoperatornorm} for the first term, $b$ and $\sigma$ satisfying a linear growth condition, \eqref{eq: integrability} and~\eqref{eq: good bound on SK} for $\overline{\eta}$ with associated constant $\overline{\delta}\in (0,1)$, we immediately obtain for every $t> 0$:
    \begin{align*}
    \mathbb{E}\big[\|\mathcal{X}_{t}\|_{\overline{\eta}}^{2}\big]&\lesssim \|S(t)\xi_g\|_{\overline{\eta}}^2+ \mathbb{E}\left[\int_{0}^{t}\|S(t-s)\xi_{K}\|_{\overline{\eta}}^{2}\,|b(X_{s})|^2\,\mathrm{d}s\right] \\
    &\hspace{0.5cm}+ \mathbb{E}\left[\int_{0}^{t}\|S(t-s)\xi_{K}\|_{\overline{\eta}}^{2}\,| \sigma(X_{s})|^2\,\mathrm{d}s\right]\\
    &\lesssim \|\xi_g\|_{\eta_g}^2\hspace{0.03cm}t^{-(\overline{\eta}-\eta_{g})^+} + \Big( 1+\sup_{s\in [0,t]}\mathbb{E}[|X_{s}|^2]\Big)\int_{0}^{t} (t-s)^{\overline{\delta} -1}\,\mathrm{d}s<\infty,
    \end{align*}
    and, therefore, the integrals in \eqref{eq: appendix lift constructed solution} are well-defined for every $t>0$. Selecting $\overline{\eta}=\max\{\eta, 1/2\}\in (\eta_*, 1-\eta_*)$, we can conclude from the linearity of the operator $\Xi$, $\Xi S(\cdot)\xi_g =g$ and the commutativity of the continuous linear functional $\Xi\hspace{0.02cm}|_{\mathcal{H}_{\overline{\eta}}}=\langle\cdot , w_{\overline{\eta}}\rangle_{\overline{\eta}}$ with $\mathcal{H}_{\overline{\eta}}$-valued Bochner and stochastic integrals in combination with $\Xi S(\cdot)\xi_K = K$ that we have $\Xi \mathcal{X}_t = X_t$ a.s. for every $t\in \R_+$, whence $\mathcal{X}_t$ satisfies \eqref{eq: markovian lift} and is, in particular $\mathcal{H}_{\eta}$-valued due to $\mathcal{H}_{\overline{\eta}}\subseteq\mathcal{H}_{\eta}$.
   
    Therefore, it remains to verify the existence of a continuous version of~$\mathcal{X}$. For the drift, the continuity of sample paths follows from the Young inequality since $\|S(\cdot)\xi_K\|_{\eta} \in L_{\mathrm{loc}}^2(\R_+)$ by \eqref{eq: good bound on SK} and $b(X) \in L^2([0,T])$ holds a.s.  
    
    For the stochastic convolution part, we apply the factorization lemma
    method from~\cite{DaPrato_Zabczyk_2014}. Namely, fix $\alpha \in (0,1)$, $T > 0$, and define
    \[
        Y_{\alpha}(t) = \int_0^t (t-s)^{-\alpha}S(t-s)\xi_K \sigma(X_s)\, \mathrm{d}B_s, \qquad t \in [0,T].
    \]
    Let us show that $Y_{\alpha} \in L^p(\Omega,\P; L^p([0,T]; \mathcal{H}_{\eta}))$ holds for each $p \in (1,\infty)$ and $\alpha$ with $0 < \alpha < \frac{\delta}{2} < \frac{1}{2}$. Then it follows $2\alpha + 1 - \delta < 1$ and thus we find $q \in (1,\infty)$ sufficiently close to $1$ such that also $2\alpha + 1 - \delta < \frac{1}{q} < 1$ holds. Let $q' \in (1,\infty)$ satisfy $\frac{1}{q} + \frac{1}{q'} = 1$. Furthermore, we may assume without loss of generality that $q' \geq p/2$. Then an application of \cite[Theorem 4.36]{DaPrato_Zabczyk_2014} combined with~\eqref{eq: good bound on SK} and Jensen's inequality yields
    \begin{align*}
        \E\left[\|Y_{\alpha}(t)\|_{\eta}^p \right]
        &\lesssim \E\left[ \left( \int_0^t (t-s)^{-2\alpha}\hspace{0.02cm} \|S(t-s)\xi_K\sigma(X_s)\|_{L_2(\R, \mathcal{H}_{\eta})}^2\, \mathrm{d}s\right)^{p/2} \right]
        \\ &\lesssim \E\left[ \left( \int_0^t (t-s)^{-2\alpha - (1-\delta)} \hspace{0.02cm}|\sigma(X_s)|^2\, \mathrm{d}s\right)^{p/2} \right]
        \\ &\lesssim t^{\frac{p}{2q}\left(1 - (2\alpha + 1 - \delta)q\right)}\,\E\left[ \left( \int_0^t \left( 1 + |X_s|^{2q'}\right)\, \mathrm{d}s\right)^{\frac{p}{2q'}} \right]
        \\ &\lesssim t^{\frac{p}{2q}\left(1 - (2\alpha + 1 - \delta)q\right)} \, t^{ \frac{p}{2q'}}\left( 1 + \sup_{s \in [0,T]}\E\left[|X_s|^{2q'} \right]\right)^{\frac{p}{2q'}}<\infty
    \end{align*}
    for $t \in [0,T]$, where $L_2(\R, \mathcal{H}_{\eta})$ denotes the space of Hilbert-Schmidt operators from $\R$ to $\mathcal{H}_{\eta}$. This proves together with \eqref{eq: integrability} that we indeed have $Y_{\alpha} \in L^p(\Omega,\P; L^p([0,T]; \mathcal{H}_{\eta}))$. Similarly, for $s\le t$ we show that
    \begin{align*}
        &\int_0^s (s-r)^{-2\alpha} \,\E\left[ \|S(t-r)\xi_K \sigma(X_r)\|_{L_2(\R, \mathcal{H}_{\eta})}^2 \right]\, \mathrm{d}r \\
        &\qquad\lesssim \int_0^s (s-r)^{-2\alpha}\hspace{0.02cm}(t-r)^{- (1-\delta)} \left( 1 + \E[|X_r|^2] \right)\, \mathrm{d}r 
        \\ &\qquad\lesssim \int_0^s r^{-2\alpha}\hspace{0.02cm}(t-s + r)^{- (1-\delta)}\, \mathrm{d}r 
        \lesssim \int_0^s r^{-2\alpha - 1 + \delta} \, \mathrm{d}r
        \lesssim s^{\delta - 2\alpha},
    \end{align*}
    where the last estimate is justified by $\delta - 2\alpha > 0$, which implies for every $t \in [0,T]$: 
    \begin{align*}
        &\ \int_0^t (t-s)^{\alpha - 1}\left( \int_0^s (s-r)^{-2\alpha} \,\E\left[ \|S(t-r)\xi_K \sigma(X_r)\|_{L_2(\R, \mathcal{H}_{\eta})}^2 \right]\, \mathrm{d}r \right)^{1/2}\, \mathrm{d}s
        \\ &\qquad  \lesssim \int_0^t (t-s)^{\alpha - 1}s^{\frac{\delta}{2} - \alpha} \, \mathrm{d}s
        \lesssim t^{\delta/2} < \infty.
    \end{align*}
    Hence, an application of \cite[Theorem 5.10]{DaPrato_Zabczyk_2014} proves that
    \[
        \int_0^t S(t-s)\xi_K \sigma(X_s)\, \mathrm{d}B_s = \frac{\sin(\alpha \pi)}{\pi}\int_0^t (t-s)^{\alpha - 1}S(t-s)Y_{\alpha}(s)\, \mathrm{d}s.
    \]
    Consequently, the factorization lemma \cite[Proposition 5.9]{DaPrato_Zabczyk_2014} applied for $E_1 = E_2 = \mathcal{H}_{\eta}$ and $r = 0$ shows that the right-hand side is continuous in $t$, which provides the desired continuous modification. This proves the assertion.
\end{proof}

\section{Proof of Lemma \ref{lemma:liftkernelL2estimates}}\label{section:appendixkernelestimates}

\begin{proof}
    The upper bound in the first estimate is an immediate consequence of \eqref{eq: K pointwise bound} for $\varepsilon\mapsto \varepsilon/2$ and $\eta_{*}\in[0,1/2)$ giving $\overline{C}_{\varepsilon}:=\|w_{\eta_* + \varepsilon/2}\|_{\eta_* + \varepsilon/2}$, whereas the boundedness of $K$ implies the claim for $\eta_{*}<0$. For the lower bound, we observe using first the Cauchy-Schwarz inequality, then the definition of $\mu$ and $K(\infty)\ge 0$, and finally Fubini's theorem:
\begin{align*}
    \int_0^h K(r)^2 \,\mathrm{d}r &\geq h^{-1} \left( \int_0^h K(r)\,\mathrm{d}r \right)^2 \geq h^{-1} \left( \int_0^h \int_{\R_+} e^{-rx} \,\mu(\mathrm{d}x) \,\mathrm{d}r \right)^2
    \\ &= h \left( \int_{\R_+} \frac{1- e^{-hx}}{hx} \,\mu(\mathrm{d}x) \right)^2 \geq h \left( \int_{0}^{1/h} \frac{1-e^{-hx}}{hx} \,\mu(\mathrm{d}x) \right)^2\\
    &\geq \left( 1 - e^{-1}\right)^2 h\hspace{0.03cm} \mu((0,1/h])^2,
\end{align*}
   where the last step is justified by $(0,1/h]\ni x\mapsto (1-e^{-hx})/(hx)$ being a decreasing function. Moreover, we obtain from $K\not\equiv 0$ being bounded away from zero on $[0,T]$ that $\|K\|_{L^2((0,h])}^{2}$ is bounded below by $\inf_{s\in [0, T]}K(s)^2\, h>0$. The former is trivial for $K(\infty)>0$, whereas on the other hand it follows from $K(\infty)=0$ and $K\not\equiv 0$ that there exists a bounded Borel measurable set $B\in\mathcal{B}_{\R_{+}}$ with $\mu(B)>0$. On $[0,T]$ we can, therefore, estimate
\begin{displaymath}
    \inf_{s\in [0,T]}K(s)=\inf_{s\in [0,T]}\int_{\R_+}e^{-xs}\,\mu(\mathrm{d}x)\ge \inf_{s\in [0,T]}\int_{B}e^{-xs}\,\mu(\mathrm{d}x)\ge e^{-bT}\mu(B)>0,
\end{displaymath}
where we defined $b:=\sup B$. Hence, combining both lower estimates implies the desired lower bound. 

For the second part, notice that also $-K'=|K'|$ is a completely monotone function, whose Bernstein measure $\mu'$ is absolutely continuous w.r.t.~$\mu_{K}$, and the corresponding density is given by $\tfrac{\mathrm{d}\mu'}{\mathrm{d}\mu_{K}}(x)=x$. Hence, we obtain $\eta_{*}'=\eta_{*}+1$,
where~$\eta_{*}'$ is defined analogously to~$\eta_*$. Even though condition~(A) might not hold for $|K'|$, since $\eta_{*}'\ge 1/2$ may occur, e.g.\ for gamma and rough Riemann-Liouville kernels, this turns out not to be an issue, since we merely need a pointwise bound in the spirit of \eqref{eq: K pointwise bound}, which becomes $|K'(t)|\lesssim \|w_{\eta_*' + \varepsilon/2}\|_{\mu', \eta_*' + \varepsilon/2}\, t^{-\eta_{*}-1-\varepsilon/2}$ in this case for fixed $\varepsilon>0$ and every $t\in (0,T]$. Hence, for $\eta_{*}\in[0,1/2)$ we can estimate for every $\varepsilon\in (0,1-2\eta_{*})$ by the monotonicity of $K$ and $|K'|$ and the first part:
    \begin{align*}
        \int_0^T |K(h + r) - K(r)|^2 \,\mathrm{d}r
        &\lesssim \int_0^h K(r)^2 \,\mathrm{d}r+ \int_h^T |K(h + r) - K(r)|^2 \,\mathrm{d}r\\
        &\lesssim \|w_{\eta_* + \varepsilon/2}\|_{\eta_* + \varepsilon/2}^{2}\,h^{1-2\eta_{*}-\varepsilon}+ h^2\int_h^T |K'(r)|^2 \,\mathrm{d}r\\
        &\lesssim C_{\eta_{*}, \varepsilon}'\left(h^{1-2\eta_{*}-\varepsilon}+ h^2\int_h^T r^{-2\eta_{*}-2-\varepsilon} \,\mathrm{d}r\right)\\
        &\lesssim C_{\eta_{*}, \varepsilon}'\left(h^{1-2\eta_{*}-\varepsilon}+ h^2 (h^{-2\eta_{*}-1-\varepsilon} +T^{-2\eta_{*}-1-\varepsilon})\right)\\
        &\lesssim C_{\eta_{*}, \varepsilon}'\,h^{1-2\eta_{*}-\varepsilon},
    \end{align*}
    where we defined $C_{\eta_{*}, \varepsilon}':=\max\big\{\|w_{\eta_* + \varepsilon/2}\|_{\eta_* + \varepsilon/2}^{2}, \|w_{\eta_*' + \varepsilon/2}\|_{\mu', \eta_*' + \varepsilon/2}^{2}\big\}$. For $\eta_{*}<-1$, the boundedness of both $K$ and $K'$ implies via the same decomposition as above that the integral is bounded by a function of the form $Ch$. Even though $K'$ is, in general, unbounded for $\eta_{*}\in (-1, 0)$, an upper bound of the same order can be achieved by selecting $\varepsilon\in (0,-2\eta_{*})$ in the second part of the above decomposition, while the first one is by the boundedness of $K$ again of linear growth. Finally, if $\int_{\R_+} (1+x)^{-\eta_*}\, \mu(\mathrm{d}x) < \infty$, then we can estimate
    \begin{align*}
        \|w_{\eta_*' + \varepsilon/2}\|_{\mu', \eta_*' + \varepsilon/2}^{2}&= 1+ \int_{\R_+}(1+x)^{-\eta_{*}-1-\varepsilon/2}x\,\mu(\mathrm{d}x)\\
        &\le \|w_{\eta_* + \varepsilon/2}\|_{\eta_* + \varepsilon/2}^{2}\le 1+ \int_{\R_+} (1+x)^{-\eta_*}\, \mu(\mathrm{d}x) < \infty.
    \end{align*}
    Consequently, we obtain $\lim_{\varepsilon\searrow 0}C_{\eta_{*}, \varepsilon}'<\infty$ and similarly for $\overline{C}_{\varepsilon}$ from the first step, whence also the choice $\varepsilon=0$ is admissible in this case as the constant $C$ is independent of $\varepsilon$ even for $\eta_{*}>-1$.
\end{proof}

\end{appendices}

\begin{footnotesize}
\bibliographystyle{siam}
\bibliography{literature}
\end{footnotesize}

\end{document}